\UseRawInputEncoding
\documentclass[12pt, reqno]{amsart}
\usepackage[margin=1in]{geometry}
\usepackage{amssymb,latexsym,amsmath,amscd,amsfonts}
\usepackage{latexsym}
\usepackage[mathscr]{eucal}
\usepackage{bm}
\usepackage{mathptmx}
\usepackage{amssymb}
\usepackage{amsthm}
\usepackage{dcolumn}
\usepackage[all]{xy}
\usepackage{enumitem}
\usepackage[utf8]{inputenc}

\def \qed {\hfill \vrule height6pt width 6pt depth 0pt}
\def\textmatrix#1&#2\\#3&#4\\{\bigl({#1 \atop #3}\ {#2 \atop #4}\bigr)}
\def\dispmatrix#1&#2\\#3&#4\\{\left({#1 \atop #3}\ {#2 \atop #4}\right)}
\newcommand{\beg}{\begin{equation}}
	\newcommand{\eeg}{\end{equation}}
\newcommand{\ben}{\begin{eqnarray*}}
	\newcommand{\een}{\end{eqnarray*}}

\newcommand{\C}{\mathbb C}

\newcommand{\N}{\mathbb N}
\newcommand{\T}{\mathbb T}

\newcommand{\E}{\mathbb E}

\newcommand{\Pe}{\mathbb P}
\newcommand{\D}{\mathbb D}
\newcommand{\G}{\mathbb G}

\newcommand{\lm}{\lambda}

\newcommand{\HS}{\mathcal{H}}

\newcommand{\al}{\alpha}
\newcommand{\DC}{\overline{\mathbb{D}}}

\newtheorem{thm}{Theorem}[section]
\newtheorem{cor}[thm]{Corollary}
\newtheorem{lem}[thm]{Lemma}

\newtheorem{prop}[thm]{Proposition}
\numberwithin{equation}{section} \theoremstyle{definition}
\newtheorem{defn}[thm]{Definition}

\newtheorem{eg}[thm]{Example}

\def\textmatrix#1&#2\\#3&#4\\{\bigl({#1 \atop #3}\ {#2 \atop #4}\bigr)}
\def\dispmatrix#1&#2\\#3&#4\\{\left({#1 \atop #3}\ {#2 \atop #4}\right)}

\begin{document}
	
	\title[Realization, interpolation and extension on the pentablock]{Realization, interpolation, extension on the pentablock and applications to $\mathbb D^2$, $\mathbb G_2$} 
	\author{SOURAV PAL AND NITIN TOMAR}
	
	\address[Sourav Pal]{Mathematics Department, Indian Institute of Technology Bombay,
		Powai, Mumbai - 400076, India.} \email{sourav@math.iitb.ac.in}
	
	\address[Nitin Tomar]{Mathematics Department, Indian Institute of Technology Bombay, Powai, Mumbai-400076, India.} \email{tomarnitin414@gmail.com}

	

	\begin{abstract}
We introduce Schur-Agler class for the pentablock $\mathbb P$ and establish a realization theorem for functions in this class. Then we prove an interpolation theorem for the pentablock with interpolating functions belonging to the corresponding Schur-Agler class. Also, we obtain an extension theorem for $\mathbb P$. Applying these results, we add a few new characterizations in the existing realization and interpolation theorems for the bidisc $\mathbb D^2$ and the symmetrized bidisc $\mathbb G_2$. Also, we give  alternative proofs to the existing extension theorems for $\mathbb D^2, \, \mathbb G_2$.
	\end{abstract}	
	
	\maketitle

	\section{Introduction}
	
	\noindent Throughout the paper, all operators are bounded linear maps acting on complex separable Hilbert spaces. We denote by $\mathcal{B}(\HS)$ the algebra of operators on a Hilbert space $\HS$. Let $\C, \D$ and $\T$ be the complex plane, the unit disc and the unit circle in the complex plane, respectively, with center at the origin. For a commuting tuple of operators $\underline{T}=(T_1, \dotsc, T_n)$, we denote by $\sigma_T(\underline{T})$ its Taylor joint spectrum. The space of all complex-valued holomorphic functions on a domain $\Omega \subseteq \C^d$ is denoted by $\text{Hol}(\Omega)$. A \textit{kernel} $k$ on a non-empty set $X$ is a positive semi-definite map $k: X \times X \to \C$ with $k(x, x) \ne 0$ for every $x \in X$. The reproducing kernel Hilbert space $\mathcal H(k)$ is the completion of the space $\text{span}\{k(\cdot,w):w\in X\}$ with inner product defined by
\[
\left\langle \sum_{i=1}^n c_i k(\cdot,w_i), \sum_{j=1}^m d_j k(\cdot,z_j)\right\rangle
=
\sum_{i=1}^n\sum_{j=1}^m c_i \overline{d}_j k(z_j,w_i).
\]
In the theory of control engineering (e.g., see \cite{Doyle, Francis}), the structured singular value of a $d \times d$ matrix A is a cost function that encodes structural information about the perturbations of $A$. For a linear subspace $E\subseteq M_d(\C)$, the structured singular value is defined by
	\[
	\mu_E(A)=\bigl(\inf\{\|X\|: X\in E,\ \det(I-AX)=0\}\bigr)^{-1}, \quad A\in M_d(\C),
	\]
	with $\mu_E(A)=0$ if $\det(I-AX) \ne 0$ for every $X \in E$. The notion of $\mu_E$ arises naturally in the analysis of systems with structured uncertainty and plays a central role in the $\mu$-synthesis problem. Given distinct points $\alpha_1 , \dots , \alpha_n \in \D$ and matrices $B_1, \dots , B_n \in M_d(\C)$, the $\mu$-synthesis problem seeks the existence an analytic map $F : \D \to M_d(\C)$ that interpolates the given data, i.e., $F(\alpha_i)=B_i$ for $1\leq i \leq n$ and satisfies $\mu_E(F(\lambda))\le 1$ for every $\lm \in \D$. This framework recovers several classical problems as special cases. For example, when $E = M_d(\C)$, $\mu_E$ becomes the operator norm and $\mu$-synthesis problem reduces to the matricial Nevanlinna-Pick interpolation problem \cite{Foias_Frazho}. Similarly, when $E=\{\alpha I : \alpha \in \C\}$, $\mu_E$ coincides with the spectral radius and $\mu$-synthesis problem becomes the spectral Nevanlinna-Pick interpolation problem \cite{AglerI, Costara2005}. The $\mu$-synthesis gives rise to domains such as the symmetrized bidisc $\G_2$, the tetrablock $\E$ and the pentablock $\Pe$. These domains were introduced in \cite{AglerYoung, Abouhajar, AglerIV}, respectively, and are defined as follows:
	\begin{enumerate}
		\item $\G_2=\{(\lm_1+\lm_2, \lm_1\lm_2) \in \C^2: |\lm_1|, |\lm_2|<1 \}$,
		
		\item $\E=\left\{(z^{(1)}, z^{(2)}, z^{(3)}) \in \C^3 : 1-z^{(1)}\al_1-z^{(2)}\al_2+z^{(3)}\al_1\al_2 \neq 0 \ \text{for all} \ \al_1, \al_2 \in \overline{\D}\right\}$,
		
		\item $\mathbb{P}=\{(a_{21}, \text{tr}(A), \det(A)) \in \C^3 : A=[a_{ij}] \in M_2(\C), \|A\|<1\}$.
	\end{enumerate}
	For a better understanding of the readers, we show how $\mu$-synthesis problem gives rise to a domain. For example, when $E$ is the subspace of all $2 \times 2$ upper triangular matrices with equal diagonal entries, it was shown in \cite{AglerIV} that the associated $\mu$-synthesis problem induces $\Pe$. Indeed, for $X=\begin{bmatrix}
		z & w \\
		0 & z
	\end{bmatrix} \in E$ and $A=[a_{ij}] \in M_2(\C)$, $\det(I-AX)=1-\text{tr}(A)z+\det(A)z^2-a_{21}w$. Consequently,
	\begin{align*}
		\mu_E(A)<1  & \iff \det(I-AX) \ne 0 \ \text{for all $X \in E$ with $\|X\| \leq 1$}\\
		& \iff 1-\text{tr}(A)z+\det(A)z^2-a_{21}w \ne 0 \ \text{for all $z \in \overline{\D}$ and $|w| \leq 1-|z|^2$}\\
		& \iff (a_{21}, \text{tr}(A), \det(A)) \in \Pe,
	\end{align*}
	where the last equivalence follows from Theorem 5.2 in \cite{AglerIV}. Thus, $A=[a_{ij}]$ is in the $\mu_E$-unit ball of $M_2(\C)$ if and only if $(a_{21}, \text{tr}(A), \det(A)) \in \Pe$. Consequently, an interpolation problem from $\D$ to $\mu_E$-unit ball can be studied through an analogous interpolation problem from $\D$ to $\Pe$ (see, e.g., Section 12 in \cite{AglerIV}). The boundedness of $\Pe$ makes it easier than the norm-unbounded $\mu_E$-unit ball, motivating the study of domains such as $\G_2, \E$ and $\Pe$.
	
	\smallskip

	The classical Nevanlinna-Pick interpolation theorem \cite{Nevanlinna, Pick} states the following: given distinct points $z_1,\dots,z_n\in\D$ and points $\lambda_1,\dots,\lambda_n\in\overline{\D}$, there exists $f \in \text{Hol}(\D)$ with $\|f\|_{\infty,\D} \leq 1$ and $f(z_i)=\lambda_i$ for $1 \leq i \leq n$ if and only if the Pick matrix 
	\[
	\bigl[(1-\lambda_i\overline{\lambda}_j)(1-z_i\overline{z}_j)^{-1}\bigr]_{i,j=1}^n
	\]
	is positive semi-definite. The interpolation theorem admits an equivalent formulation in terms of the Szeg\H{o} kernel $k_S(z,w) = (1 - z\overline{w})^{-1}$ on the unit disc $\D$, by expressing the Pick matrix in the form
	$ \bigl[(1 - \lambda_i \overline{\lambda}_j)k_S(z_i, z_j)\bigr]_{i,j=1}^n $.
	  This kernel-based formulation provides a natural framework for interpolation on more general domains in $\C^d$. For example, Abrahamse \cite{AbrahamseI} studied a similar interpolation problem on an $m$-holed planar domain $R$, and obtained a necessary and sufficient condition for the existence of an interpolating function in terms of a family of Pick matrices associated with a corresponding family of reproducing kernels on $R$. A more general approach to interpolation via families of kernels can be found in \cite{Jury}. A criterion for Nevanlinna-Pick interpolation problem on the bidisc $\D^2$ was given in \cite{Agler_McCarthyI, Agler_McCarthy} as follows: let $z_1, \dotsc, z_n \in \D^2$ be distinct points, and let $\lambda_1, \dotsc, \lambda_n \in \overline{\D}$. Then there is a holomorphic function $f : \D^2 \to \overline{\D}$ satisfying $f(z_i) = \lambda_i$ for $1 \leq i \leq n$ if and only if
	\[
	\bigl[(1 - \lambda_i \overline{\lambda_j})\, k(z_i, z_j)\bigr]_{i,j=1}^n 
	\]
	is positive semi-definite for every kernel $k$ on $\D^2$ that is holomorphic in the first variable and conjugate holomorphic in the second variable, and for which the coordinate multiplication operators are contractions on the reproducing kernel Hilbert space $\mathcal{H}(k)$. In a general setting, one considers the following interpolation problem: for distinct points $z_1, \dotsc, z_n$ in a domain $\Omega \subset \C^d$ and points $\lm_1, \dotsc, \lm_n$ in $\DC$, does there exist a holomorphic function $f:\Omega \to \C$ satisfying $f(z_i)=\lm_i$ for $1 \leq i \leq n$? Needless to mention, an interpolating function is supposed to satisfy an appropriate norm condition based on the geometry of the domain $\Omega$. 
	
	\smallskip 
	
	It is evident from the literature (mentioned above) that a first step towards solving the Nevanlinna-Pick interpolation problem for a domain is to find a useful characterization of the interpolating functions of the domain. Such a characterization is typically called a \textit{realization formula} (or, \textit{realization theorem}) for that domain. Two popular classes of functions play important role here, the Schur class and Schur-Agler class. The \textit{Schur class} of functions of a domain $\Omega \subset \mathbb C^n$, which is denoted by $S(\Omega)$, is defined by $S(\Omega)=\{f \in \text{Hol}(\Omega) : \|f\|_{\infty, \Omega} \leq 1\}$, where $\|f\|_{\infty, \Omega}=\sup_{z \in \Omega}|f(z)|$. In particular, the Schur class $S(\D)$ of the unit disc $\D$ is the collection of all holomorphic functions $f$ on $\D$ with $\|f\|_{\infty, \D} \leq 1$. Similarly, one defines the Schur class $S(\D^2)$ for the bidisc. An important subclass of $S(\D^2)$ is the \textit{Schur-Agler} class  given by
\[
SA(\D^2)=\{g \in \text{Hol}(\D^2) : \|g(T_1, T_2)\| \leq 1 \ \text{for all commuting pairs $(T_1, T_2)$ of strict contractions}\}.
\]
In the classical Nevanlinna-Pick interpolation theorem on $\D$, the interpolating function $f$ belongs to the Schur class $S(\D)$ of $\D$. The corresponding realization theorem (see \cite{Agler_McCarthy}) states that $f \in S(\D)$ if and only if there exist a Hilbert space $\HS$ and a unitary 
	\[
	V=\begin{bmatrix} A & B \\ C & D \end{bmatrix}: \C \oplus \HS \to \C \oplus \HS
	\]
	such that $f(z)=A+zB(I_\HS-zD)^{-1}C$. Realization formulas for Schur 
	class functions have been established for various domains, including the annulus \cite{Drit2007_I} and the bidisc \cite{Agler1990, Agler_McCarthy}. In case of the polydisc $\mathbb{D}^d$ for $d \geq 3$, not every Schur class function admits a realization formula. However, its subclass, known as the Schur-Agler class, possess a realization formula, e.g., see \cite{Agler_McCarthy}. This idea has been further extended to a more abstract framework in \cite{Drit2007_I, Drit2007_II, Ball_Guerra}, where a domain $\Omega$ is replaced by an arbitrary set $X$ and the Schur class is replaced by a class of functions determined by a prescribed family of test functions on $X$. Moreover, the realization and interpolation theorems for the symmetrized bidisc $\G_2$ were obtained in \cite{AglerYoung2017, Tirtha_Sau} with different methods. For the tetrablock $\E$, the realization and interpolation theorems were obtained by the authors of \cite{Jain}.
	
	\smallskip

	The interpolation theorem often serves as a fundamental tool in obtaining extension theorem for a domain. This phenomenon has been observed in the case of the bidisc \cite{Agler_McCarthy_2003}, the symmetrized bidisc \cite{Tirtha_Sau} and the tetrablock \cite{Jain}. Let $V$ be a nonempty subset of a domain $\Omega \subset \mathbb{C}^d$. A function $f: V \to \mathbb{C}$ is said to be holomorphic if it extends to a holomorphic function in a neighbourhood of $V$. A natural question is if any such function admits a norm-preserving extension to $\Omega$. This is called the (norm-preserving) \textit{extension problem} for a domain $\Omega$. 
	
\smallskip 
	
The main aim of this article is to achieve realization, interpolation and extension theorems for the pentablock $\Pe$. As a consequence, we recover the (previously proved) same results for the bidisc $\D^2$ and the symmetrized bidisc $\G_2$. The realization and interpolation theorems on $\Pe$ are established in Section \ref{sec_penta}, and the extension theorem for $\Pe$ is proved in Section \ref{sec_ext}. Also, we present in Section \ref{sec_appl} the corresponding results for $\D^2$ and $\G_2$. To begin with, we introduce the class $\mathfrak{M}_\Pe$, which is the collection of all commuting operator triples $\underline{T}=(T_1, T_2, T_3)$ satisfying
	\[
	\|T_2\|<2, \quad \|(2\al T_3-T_2)(2-\al T_2)^{-1}\|<1 \quad \text{and} \quad \|(1-|\al|^2)T_1(I-\al T_2-\al^2T_3)^{-1}\|<1
	\]
	for all $\al \in \DC$. The motivation for introducing such a class of operator triples is discussed in Section \ref{sec_penta}. Also, it is shown in the same section that the joint spectrum $\sigma_T(\underline{T}) \subseteq \Pe$ for all $\underline{T} \in \mathfrak{M}_\Pe$ ensuring the existence of a functional calculus $f(\underline{T})$ for all $f \in \text{Hol}(\Pe)$. Next, we define the Schur-Agler class $SA(\Pe)$ of $\Pe$ as
	\[
	SA(\Pe)=\left\{f \in \text{Hol}(\Pe) : \|f(\underline{T})\| \leq 1 \ \text{for all} \ \underline{T} \in \mathfrak{M}_\Pe\right\}.
	\]
We present in Theorem \ref{thm_realization_P} a realization theorem for $\Pe$, which is precisely the characterizations of functions belonging to Schur-Agler class $SA(\Pe)$. Furthermore, we obtain in Theorem \ref{thm_interpolation_P} necessary and sufficient conditions for a function $f \in SA(\Pe)$ to solve an $n$-points interpolation problem on $\Pe$ with values in $\DC$. This is referred to as a Nevanlinna-Pick type interpolation theorem for $\Pe$. To obtain an extension theorem for $\Pe$, we first introduce in Section \ref{sec_ext} the class $Q\Pe$ consisting of all commuting Hilbert space operator triples $\underline{T}=(T_1, T_2, T_3)$ with $\sigma_T(\underline{T}) \subseteq \Pe$ and satisfying
\[
\|T_2\| \leq 2, \quad \|(2\al T_3-T_2)(2-\al T_2)^{-1}\| \leq 1 \quad \text{and} \quad \|(1-|\al|^2)T_1(I-\al T_2-\al^2T_3)^{-1}\| \leq 1
\]
for all $\al \in \DC$. The class $Q\Pe$ is referred to as the \textit{quantum pentablock}, in the spirit of Mittal’s work \cite{Mittal} on quantized domains in $\C^d$. Note that $Q\Pe$ strictly contains $\mathfrak{M}_\Pe$. For $W \subseteq \Pe$, we define the notion of `subordinate to $W$' suitably in the setting of $\Pe$ and consider the operator triples in $Q\Pe$ that are subordinate to $W$. Finally, we prove our extension result as Theorem \ref{thm_ext_P}, which states that if $W \subseteq \Pe$ and $f$ is a bounded non-zero function on 
$W$ that admits a holomorphic extension to a neighbourhood of $W$, then there exists $g \in H^\infty(\Pe)$ with $g|_W = f$, $\|g\|_{\infty,\Pe} = \|f\|_{\infty,W}$ and $g/\|f\|_{\infty,W} \in SA(\Pe)$ if and only if  
$
\|f(\underline{T})\| \le \|f\|_{\infty,W}
$ 
for every $\underline{T} \in Q\Pe$ subordinate to $W$. The pentablock has several interesting geometric properties, e.g., see \cite{AglerIV, Pal_penta} and the references therein. In particular, one such property is its connection with the domains $\D^2$ and $\G_2$, as described below.
	
	\begin{thm}[\cite{Pal_penta}, Section 3]\label{thm_connect_P}
		A pair $(z^{(1)}, z^{(2)}) \in \D^2$ if and only if $(z^{(1)}, 0, z^{(2)}) \in \Pe$. Also, $(z^{(1)}, z^{(2)}) \in \G_2$ if and only if $(0, z^{(1)}, z^{(2)}) \in \Pe$.
	\end{thm}
	Capitalizing on Theorem \ref{thm_connect_P}, we obtain in Section \ref{sec_appl} the realization and interpolation theorems for $\D^2$ as in \cite{Agler_McCarthyI, Agler_McCarthy} and $\mathbb G_2$ as in \cite{AglerYoung2017, Tirtha_Sau}, with our characterizations formulated in terms of functions in $SA(\Pe)$. In addition, the extension theorems on $\D^2$ and $\G_2$ established as in \cite{Agler_McCarthy_2003} and \cite{Agler2019, Tirtha_Sau}, respectively, are recovered from the extension theorem on $\Pe$. This provides a unified approach to the realization, interpolation and extension theorems for $\D^2$ and $\G_2$ through $\Pe$.

	\section{Realization and interpolation theorem on the pentablock}\label{sec_penta}
	
	\noindent In this section, we introduce a Schur-Agler type class for the pentablock $\Pe$ and present a realization theorem for functions belonging to this class. Furthermore, we prove an interpolation theorem for $\Pe$, where the interpolating function is a member of the Schur-Agler class for $\Pe$. To begin with, we recall from \cite{AglerIV} the characterization of $\Pe$ given as follows:
	\[
	\mathbb P= \left\{(z^{(1)}, z^{(2)}, z^{(3)}) \in \mathbb{C} \times \mathbb{G}_2 :
	\sup_{\alpha \in \mathbb{D}}
	\left|\frac{(1-|\al|^2)z^{(1)}}{1- z^{(2)}\al+z^{(3)}\al^2}\right|<1 \right\},
	\]
	where $\G_2=\{(\lm_1+\lm_2, \lm_1\lm_2): \lm_1, \lm_2 \in \D\}$. The closure of $\G_2$ is denoted by $\Gamma$. The following characterization of $\G_2$ was obtained by the authors of \cite{Agler2004_II} in terms of a family of functions parameterized over $\DC$:
	\begin{equation}\label{eqn_G_2}
		(z^{(2)}, z^{(3)}) \in \G_2 \ \  \text{if and only if}	\ |z^{(2)}|<2, \ \ |\Phi_\alpha(z^{(2)}, z^{(3)})|<1 \ \text{for all} \ \al \in \DC, 
	\end{equation}
	where $\Phi_\alpha(z^{(2)}, z^{(3)})=(2\al z^{(3)}-z^{(2)})\slash (2-\al z^{(2)})$. Hence, $z=(z^{(1)}, z^{(2)}, z^{(3)}) \in \Pe$ if and only if 
	\begin{equation}\label{eqn_P}
		|z^{(2)}|<2, \quad |\Phi_\alpha(z^{(2)}, z^{(3)})|<1 \quad \text{and} \quad |\psi_{\al}(z^{(1)}, z^{(2)}, z^{(3)})|<1
	\end{equation}
	for all $\al \in \DC$, where 
	\[
	\psi_{\al}(z^{(1)}, z^{(2)}, z^{(3)})=\frac{(1-|\al|^2)z^{(1)}}{1- z^{(2)}\al+z^{(3)}\al^2}.
	\]
	Also, $\G_2$ and $\Pe$ are quasi-balanced domains, i.e., $(rw^{(2)}, r^2w^{(3)}) \in \G_2$ and $(rz^{(1)}, rz^{(2)}, r^2z^{(3)}) \in \Pe$ for $0 \leq r \leq 1, (w^{(2)}, w^{(3)}) \in \G_2$ and $(z^{(1)}, z^{(2)}, z^{(3)}) \in \Pe$. Using the description of $\Pe$ as in \eqref{eqn_P}, we define for $z = (z^{(1)}, z^{(2)}, z^{(3)}) \in \Pe$ the following functions:
	\begin{align}\label{eqn_J(z)}
		\mathrm{J}(z), \ \mathrm{j}(z): \DC \to \C, \quad \mathrm{J}(z)(\al)=\psi_{\al}(z^{(1)}, z^{(2)}, z^{(3)}) \quad \text{and} \quad \mathrm{j}(z)(\al) =\Phi_\al(z^{(2)}, z^{(3)}).
	\end{align}
	Clearly, $\mathrm{J}(z)$ and $\mathrm{j}(z)$ are continuous functions on $\DC$, and $\|\mathrm{J}(z)\|_{\infty, \DC}, \|\mathrm{j}(z)\|_{\infty, \DC}<1$ for $z \in \Pe$. We now introduce an operator-theoretic analog of the inequalities provided in \eqref{eqn_P}. Let $(T_1, T_2, T_3)$ be a commuting triple of operators acting on a Hilbert space $\HS$ and let $\|T_2\|<2$. Evidently, the operator $(2-\alpha T_2)$ is invertible for all $\al \in \DC$. Additionally, assume that
	\[
	\|\Phi_\al(T_2, T_3)\|=\|(2\al T_3-T_2)(2-\al T_2)^{-1}\|<1 
	\]
	for all $\al \in \DC$. We consider the class of commuting pairs of operators given by 
	\[
	\mathfrak{M}_{\G_2}=\left\{(T_2, T_3): \|T_2\|<2, \|\Phi_\al(T_2, T_3)\|<1 \ \text{for all} \ \al \in \DC\right\},
	\]
	which was studied by the authors of \cite{AglerYoung2017} (also see \cite{Tirtha_Sau}). The inequalities in $\mathfrak{M}_{\G_2}$ can be interpreted as an operator theoretic analog of inequalities from \eqref{eqn_G_2} in which the scalars are replaced by commuting operator pairs subjected to similar norm bounds. The class $\mathfrak{M}_{\G_2}$ contains $\G_2$ in the sense that $(z^{(2)}I, z^{(3)}I) \in \mathfrak{M}_{\G_2}$ for all $(z^{(2)}, z^{(3)}) \in \G_2$, which follows from \eqref{eqn_G_2}. Let $(T_2, T_3) \in \mathfrak{M}_{\G_2}$ and $(w^{(2)}, w^{(3)}) \in \sigma_T(T_2, T_3)$. By projection property of the joint spectrum, $w^{(2)} \in \sigma(T_2)$ and so, $|w^{(2)}|<2$ since $\|T_2\|<2$. For $\al \in \DC$, we have that $\|\Phi_\al(T_2, T_3)\|<1$. It now follows from spectral mapping principle that
	$
	\Phi_\al(w^{(2)}, w^{(3)}) \in \Phi_\al(\sigma_T(T_2, T_3))=\sigma(\Phi_\al(T_2, T_3))\subseteq \D.
	$
	By \eqref{eqn_G_2}, $(w^{(2)}, w^{(3)}) \in \G_2$. Consequently, $\sigma_T(T_2, T_3) \subseteq \G_2$ and so, one can define the operator $f(T_2, T_3)$ for all $f \in \text{Hol}(\G_2)$ and $(T_2, T_3) \in \mathfrak{M}_{\G_2}$. Accordingly, the Schur-Agler class for $\G_2$ can be defined as
	\[
	SA(\G_2)=\left\{f \in \text{Hol}(\G_2): \|f(\underline{T})\| \leq 1 \ \text{for all} \ \underline{T} \in \mathfrak{M}_{\G_2}\right\},
	\]
	which is a subset of the Schur class $S(\G_2)=\{f \in \text{Hol}(\G_2): |f(z)| \leq 1 \ \text{for all} \ z \in \G_2\}$. In fact, the Schur and Schur-Agler classes of $\G_2$ are same (e.g., see \cite{AglerYoung2017, AglerYoung, Tirtha_Sau} and Theorem 1.5 in \cite{AglerYoung2003}). We briefly sketch an argument here from the literature. Let $f \in S(\G_2)$ and $\underline{T}=(T_2, T_3) \in \mathfrak{M}_{\G_2}$. For $0<r<1$, define $f_r(z^{(2)}, z^{(3)})=f(rz^{(2)}, r^2z^{(3)})$, which is holomorphic on a neighbourhood of the closed symmetrized bidisc $\Gamma$. We have by Theorem 1.5 in \cite{AglerYoung2003} that $(T_2, T_3)$ has $\Gamma$ as a spectral set, i.e., $\|g(T_2, T_3)\| \leq \|g\|_{\infty, \Gamma}$ for every function $g$ holomorphic in a neighbourhood of $\Gamma$. Then 
	\[
	\|f_r(T_2, T_3)\| \leq \sup \{|f_r(z^{(1)}, z^{(2)})|: (z^{(2)}, z^{(3)}) \in \Gamma\} \leq \sup \{|f(z^{(1)}, z^{(2)})|: (z^{(2)}, z^{(3)}) \in \G_2\}  \leq 1.
	\]
	Letting $r \to 1$ yields that $\|f(T_2, T_3)\| \leq 1$ and so, $f \in SA(\G_2)$. Thus, $SA(\G_2)=S(\G_2)$.

	\smallskip 	
	
	The domain $\Pe$ is fibered over $\G_2$, that is, the last two components of a scalar triple in $\Pe$ belong to $\G_2$. In the same spirit, we construct the class of commuting triples $(T_1, T_2, T_3)$ associated with $\Pe$, where the last two components $(T_2, T_3) \in \mathfrak{M}_{\G_2}$. For this purpose, let $\al \in \DC$ and consider the polynomial
	$
	p(z^{(2)}, z^{(3)})=1-z^{(2)}\al+z^{(3)}\al^2.
	$
	By \eqref{eqn_G_2}, $p$ is non-vanishing on $\G_2$ and, in particular, on $\sigma_T(T_2, T_3)$ for every $(T_2, T_3) \in \mathfrak{M}_{\G_2}$. By spectral mapping principle, the operator 
	$
	p(T_2, T_3)=I-\al T_2+\al^2T_3
	$
	is invertible for every $\al \in \DC$. Consequently, one can define the operator 
	\[
	\psi_{\al}(T_1, T_2, T_3)=(1-|\al|^2)T_1(I-\al T_2+\al^2T_3)^{-1}
	\]
	for all $\al \in \DC$ and commuting triples $(T_1, T_2, T_3)$ of Hilbert space operators such that $(T_2, T_3) \in \mathfrak{M}_{\G_2}$. In light of this, we introduce the class $\mathfrak{M}_\Pe$ consisting of all commuting triples $\underline{T}=(T_1, T_2, T_3)$ of Hilbert space operators such that 
	for all $\al \in \DC$,
	\[
	\|T_2\| <2, \quad \|\Phi_{\al}(T_2, T_3)\|<1 \quad \text{and} \quad \|\psi_{\al}(T_1, T_2, T_3)\|<1.
	\]
	We mention the following properties of $\mathfrak{M}_\Pe$, which will be used frequently throughout the section.
	\begin{enumerate}[leftmargin=*]
		
		\item $\mathfrak{M}_\Pe$ contains $\Pe$ in the sense that $(z^{(1)}I, z^{(2)}I, z^{(3)}I) \in \mathfrak{M}_\Pe $ for all $(z^{(1)}, z^{(2)}, z^{(3)}) \in \Pe$.
		
		\item $\mathfrak{M}_\Pe$ is closed under taking adjoint, i.e., $(T_1, T_2, T_3) \in \mathfrak{M}_\Pe$ if and only if $(T_1^*, T_2^*, T_3^*) \in \mathfrak{M}_\Pe$.
		
		\item Since $\Pe$ is a $(1, 1, 2)$-quasi-balanced domain, it follows that 
		$(rT_1, rT_2, r^2T_3) \in \mathfrak{M}_\Pe$ for $0 \leq r \leq 1$ and $(T_1, T_2, T_3) \in \mathfrak{M}_\Pe$.
		
		\item For a commuting operator triple $(T_1, T_2, T_3) \in \mathfrak{M}_\Pe$, it is evident that $(T_2, T_3) \in \mathfrak{M}_{\G_2}$.
		
		\item It is not difficult to prove that $\sigma_T(T_1, T_2, T_3) \subseteq \Pe$ for all $(T_1, T_2, T_3) \in \mathfrak{M}_\Pe$. To see this, let $\underline{T}=(T_1, T_2, T_3) \in \mathfrak{M}_\Pe$ and let $(w^{(1)}, w^{(2)}, w^{(3)}) \in \sigma_T(\underline{T})$. By projection property of the joint spectrum, $(w^{(2)}, w^{(3)}) \in \sigma_T(T_2, T_3)$. As $(T_2, T_3) \in \mathfrak{M}_{\G_2}$, it follows that $(w^{(2)}, w^{(3)}) \in \G_2$. For $\alpha \in \DC$, $\|\psi_{\alpha}(T_1, T_2, T_3)\|<1$. We have by spectral mapping principle that
		$
		\psi_{\alpha}(w^{(1)}, w^{(2)}, w^{(3)}) \in \psi_{\alpha}(\sigma_T(T_1, T_2, T_3))=\sigma(\psi_{\alpha}(T_1, T_2, T_3) \subseteq \D 
		$
		and by \eqref{eqn_P}, $(w^{(1)}, w^{(2)}, w^{(3)}) \in \Pe$. 
	\end{enumerate}
	The above discussion shows that the functional calculus $f(\underline{T})$ is well defined for all $f \in \mathrm{Hol}(\mathbb{P})$ and $\underline{T} \in \mathfrak{M}_\mathbb{P}$. In this direction, we consider the following class: 
	\[
	SA(\Pe)=\left\{f \in \text{Hol}(\Pe) : \|f(\underline{T})\| \leq 1 \ \text{for all} \ \underline{T} \in \mathfrak{M}_\Pe\right\},
	\]
	which we call the Schur-Agler class for the pentablock. Evidently, $SA(\Pe) \subseteq S(\Pe)$.  For a non-empty set $Y$, a map $f: Y \times Y \to \C$ is said to be \textit{self-adjoint} on $Y$ if $f(z, w)=\overline{f(w, z)}$ for all $z, w \in Y$, and $f$ is called \textit{positive semi-definite} (written $f \succcurlyeq 0$) if
	$
	\sum_{i,j=1}^n \bar{c}_ic_j f(y_i, y_j) \geq 0$
	for every $\{y_1,\dotsc, y_n\} \subseteq Y, \{c_1,\ldots,c_n\} \subset \C$ and $n \in \N$.  A positive semi-definite function $k: \Pe \times \Pe \to \C$ is referred to as a \textit{weak kernel} on $\Pe$. In addition, if $k(z, z) \ne 0$ for all $z \in \Pe$, we say that $k$ is a \textit{kernel} on $\Pe$. For a subset $F$ of $\Pe$, a function $\xi: F \times F \to C(\DC)^*$ is said to be a positive kernel with values in the dual space $C(\DC)^*$ if the following holds:
	\[
	\overset{n}{\underset{i, j=1}{\sum}}\overline{c}_ic_j\xi(z_i, z_j)(f_i \overline{f}_j) \geq 0
	\]  
	for every $n \in \N, \{z_1, \dotsc, z_n\} \subseteq F, \{c_1, \dotsc, c_n\} \subset \C$ and $\{f_1, \dotsc, f_n\} \subseteq C(\DC)$. We denote by $C(\DC)_F^+$ the set of all positive kernels with values in $C(\DC)^*$, and the set of all positive semi-definite functions $\delta: F \times F \to \C$ is denoted by $\C_F^+$.
	
	\begin{defn}\label{defn_AK_P}
		A kernel (weak kernel) $k: \Pe \times \Pe \to \C$ is said to be \textit{admissible} (\textit{weakly admissible}) if the following hold:
		\begin{enumerate}
			\item $\displaystyle \left(1-\frac{z^{(2)}\overline{w}^{(2)}}{4}\right)k(z, w) \succcurlyeq 0$; \smallskip 
			\item $\left(1-\Phi_\al(z^{(2)}, z^{(3)})\overline{\Phi_\al( w^{(2)}, w^{(3)})}\right)k(z, w) \succcurlyeq 0$ for all $\al \in \DC$; \smallskip 
			\item $\left(1-\psi_{\al}(z^{(1)}, z^{(2)}, z^{(3)})\overline{\psi_{\al}(w^{(1)}, w^{(2)}, w^{(3)})}\right)k(z, w) \succcurlyeq 0$ for all $\al \in \DC$.
		\end{enumerate}
		The class of admissible kernels on $\Pe$ is denoted by $AK(\Pe)$.	
	\end{defn}
	
	As a first step towards establishing a realization theorem for $SA(\Pe)$, we prove the following result. Although the argument is well-known and appears in the literature (see, for example, \cite{Agler_McCarthy, Drit2007_I, Drit2007_II, Jain}), we include its proof here for completeness and to adapt it to the present setting.
	
	\begin{prop}\label{prop_prelim_I_P}
		Let $F$ be a subset of $\Pe$. If $\xi \in C(\DC)_F^+$, then there exist a Hilbert space $\HS$, a function $L: F \to \mathcal{B}(C(\DC),\HS)$ and a unital $*$-representation $\rho: C(\DC) \to \mathcal{B}(\HS)$ such that
		\[
		\xi(z, w)(f\overline{g})=\langle L(z)f, L(w)g \rangle \quad \text{and} \quad L(z)(fg)=\rho(f)L(z)(g)
		\]	
		for all $f, g \in C(\DC)$ and $z, w  \in F$. 
	\end{prop}
	
	\begin{proof}
		Consider $\xi': (F \times C(\DC)) \times (F \times C(\DC)) \to \C$ given by $\xi'((z, h_1), (w, h_2))=\xi(z, w)(h_1\overline{h}_2)$. It is not difficult to see that $\xi'$ is a positive semi-definite function on $F \times C(\DC)$. It follows from Theorem 2.53 in \cite{Agler_McCarthy} that there is a Hilbert space $\mathcal{H}$ and a vector-valued map $\eta: F \times C(\DC) \to \mathcal{H}$ such that 
		\[
		\langle \eta(z, h_1), \eta(w, h_2)\rangle_{\mathcal{H}}=\xi'((z, h_1), (w, h_2))=\xi(z, w)(h_1\overline{h}_2)
		\]
		for all $h_1, h_2 \in C(\DC)$ and  $z, w \in F$. Indeed, one can choose the Hilbert space $\HS=\overline{\text{span}}\{\eta(z, h): z \in F, h \in C(\DC)\}$. Let us define
		$L: F \to \mathcal{B}(C(\DC), \HS)$ as $L(z)(h)=\eta(z, h)$. Then 
		\[
		\xi(z, w)(f\overline{g})=\langle L(z)f, L(w)g \rangle  \ \text{and so,}  \ \|L(z)(f)\|^2=\|\eta(z, f)\|^2=\|\xi(z, z)(f\bar{f})\|^2 \leq \|\xi(z, z)\| \|f\|^2_{\infty, \DC}
		\]
		for all $z, w \in F$ and $f, g \in C(\DC)$. Consider the map $\rho: C(\DC) \to \mathcal{B}(\HS)$ given by $\rho(h_1)\eta(z, h_2)=\eta(z, h_1h_2)$. A simple computation shows that $\rho$ is a unital $*$-representation such that $\rho(f)L(z)(g)=L(z)(fg)$. The proof is now complete.
	\end{proof}
	
	We now present a description of self-adjoint functions on $\Pe$ such that their product with every admissible kernel on $\Pe$ is positive semi-definite. Let $F$ be a finite subset of $\Pe$ with cardinality $|F|$. Let $\mathfrak{P}_F$ be the collection of matrices of the form 
	\[
	\left[\xi(z, w)(1-\mathrm{J}(z)\overline{\mathrm{J}(w)})+\nabla(z, w)(1-\mathrm{j}(z)\overline{\mathrm{j}(w)})+(1-(z^{(2)}\overline{w}^{(2)}\slash 4))\delta(z, w)\right]_{z, w \in F},
	\]
	where $\xi, \nabla \in C(\overline{\D})^+_F$ and $\delta \in \C_F^+$. Following the same arguments as in Lemma 3.4 of \cite{Drit2007_I}, we observe that $\mathfrak{P}_F$ is a closed cone in $M_{|F|}(\C)$, the space of $|F| \times |F|$ matrices. Also, $\mathfrak{P}_F$ has non-empty interior since any positive semi-definite matrix in $M_{|F|}(\C)$ belongs to $\mathfrak{P}_F$. To see this, let $P=[P(z, w)]_{z, w \in F}$ be a positive semi-definite matrix. Define $\delta(z, w)=P(z, w)\delta_0(z, w)$, where $\displaystyle \delta_0(z, w)=\frac{1}{1-(z^{(2)}\overline{w}^{(2)}\slash 4)}$ for $z, w \in F$. A simple computation shows that $\delta \in \C_F^+$ and so, 
	\[
	[P(z, w)]_{z, w \in F}=[\delta(z, w)(1-(z^{(2)}\overline{w}^{(2)}\slash 4))]_{z, w \in F} \in \mathfrak{P}_F.
	\]
	In particular, the $|F| \times |F|$ matrix with all entries equal to $1$ is in $\mathfrak{P}_F$. For further details in a more general setting, an interested reader is referred to Section 3 of \cite{Drit2007_I} and Section 5 of \cite{Drit2007_II}.
	
	\begin{thm}\label{thm_sa_P}
		Let $g: \Pe \times \Pe \to \C$ be a self-adjoint function. If $gk \succcurlyeq 0$ for all $k \in AK(\Pe)$, then  there exist $\xi, \nabla \in C(\DC)^+_\Pe$ and $\delta \in \C_\Pe^+$ such that 
		\[
		g(z, w)=\xi(z, w)(1-\mathrm{J}(z)\overline{\mathrm{J}(w)})+\nabla(z, w)(1-\mathrm{j}(z)\overline{\mathrm{j}(w)})+(1-(z^{(2)}\overline{w}^{(2)}\slash 4))\delta(z, w)
		\]
		for all $z, w \in \Pe$. 
	\end{thm}
	
	\begin{proof}
		Let $gk \succcurlyeq 0$ for all $k \in AK(\Pe)$, and let $F=\{z_1, \dotsc, z_n\}$ be a subset of $\Pe$. We show that $G=[g(z_i, z_j)]_{i, j=1}^n \in \mathfrak{P}_F$. Let if possible, $G \notin \mathfrak{P}_F$. Since $\mathfrak{P}_F$ is a closed cone with non-empty interior, an application of Hahn-Banach theorem ensures the existence of a linear functional $L$ on $M_{n}(\C)$ such that $L(G)<0$ and $L(M) \geq 0$ for all $M \in \mathfrak{P}_F$. In fact, the functional $L$ can be chosen such that $L(A)=\text{tr}(AC)$ for all self-adjoint matrices $A$, where $C=[C_{ij}]_{i, j=1}^n$ is a fixed self-adjoint matrix. Given $\{c_1, \dotsc, c_n\} \subseteq \C$, we have that $N=[\overline{c}_ic_j] \in \mathfrak{P}_F$ and thus, it follows that
		\[
		L(N)=\overset{n}{\underset{i, j=1}{\sum}}\overline{c}_ic_jC_{ji} \geq 0.
		\]
		Consequently, the map $C_*: F \times F \to \C$ given by $C_*(z_i, z_j)=C_{ji}$ is positive semi-definite and so, $C_*$ is a weak kernel on $F$. Consider the matrices in $\mathfrak{P}_F$ given by
		\begin{align*}
			\widetilde{N}&=\left[\overline{c}_ic_j(1-(z_i^{(2)}\bar{z}_j^{(2)}\slash 4))\right]_{i, j=1}^n, \quad N_\alpha=\left[\overline{c}_ic_j(1-\Phi_{\alpha}(z_i^{(2)}, z_i^{(3)})\overline{\Phi_{\alpha}(z_j^{(2)}, z_j^{(3)})})\right]_{i, j=1}^n \ \text{and} \\
			M_{\al}&=\left[\overline{c}_ic_j(1-\psi_{\al_1}(z_i)\overline{\psi_{\al}(z_j)})\right]_{i, j=1}^n
		\end{align*}
		for $\al \in \DC$. An application of the positive semi-definiteness of $L(\widetilde{N}), L(N_\al)$ and $L(M_{\al})$ yields that $C_*$ is a weakly admissible kernel on $F$. One can define $C_*$ on $\Pe$ by putting $C_*(z, w)=0$ for all $(z, w) \notin F \times F$. For $\epsilon >0$, it is evident that $\epsilon k+C_* \succcurlyeq 0$ for all $k \in AK(\Pe)$. By hypothesis, $g(\epsilon k +C_*) \succcurlyeq 0$ on $\Pe$, and in particular, on $F$. Consequently, $gC_* \succcurlyeq 0$ on $F$. Then
		\[
		L(G)=\text{tr}(GC)=\overset{n}{\underset{i, j=1}{\sum}}g(z_i, z_j)C_*(z_i, z_j) \geq 0,
		\]
		contradicting $L(G)<0$. Hence, $G \in \mathfrak{P}_F$ and thus, it follows that 
		there exist $\xi, \nabla \in C(\DC)^+_F$ and $\delta \in \C_F^+$ such that 
		$
		g(z, w)=\xi(z, w)(1-\mathrm{J}(z)\overline{\mathrm{J}(w)})+\nabla(z, w)(1-\mathrm{j}(z)\overline{\mathrm{j}(w)})+(1-(z^{(2)}\overline{w}^{(2)}\slash 4))\delta(z, w)
		$
		for all $z, w \in F$. Thus, we obtain the desired representation of $g$ on every finite subset of $\Pe$. Since a representation on a larger set restricts naturally to any smaller subset, an application of the same reasoning as in Theorem 11.5 of \cite{Agler_McCarthy}, together with Kurosh’s theorem (see \cite{V_Arkh}, Page 7) gives the asserted conclusion.
	\end{proof}
	
	Following the terminologies in \cite{Ball}, we define a unitary colligation in the pentablock setting.
	
	\begin{defn} A function $f: \Pe \to \C$ is said to be associated to a \textit{unitary colligation} if there exist  
		\begin{enumerate}
			\item Hilbert spaces $\HS_1, \HS_2$ and $\HS_3$,
			\item unital $*$-representations $\rho_1: C(\DC) \to \mathcal{B}(\HS_1), \rho_2: C(\DC) \to \mathcal{B}(\HS_2)$ and 
			\item a unitary $V=\begin{bmatrix} A & B \\ C & D \end{bmatrix}: \C \oplus \HS \to \C \oplus \HS$, where $\HS=\HS_1 \oplus \HS_2 \oplus \HS_3$
		\end{enumerate}
		such that $f(z)=A+BY(z)(I_\HS-DY(z))^{-1}C$, where
		\[
		Y(z)=\begin{bmatrix}
			\rho_1(\mathrm{J}(z)) & 0 & 0 \\
			0 & \rho_2(\mathrm{j}(z)) & 0 \\
			0 & 0 & (z^{(2)}\slash 2)I_{\HS_3}
		\end{bmatrix} \quad (z=(z^{(1)}, z^{(2)}, z^{(3)}) \in \Pe).
		\]
		The class of functions $f: \Pe \to \C$ which is associated to a unitary colligation is denote by $UC(\Pe)$.
	\end{defn}
	It follows from the definition of $UC(\Pe)$ that $UC(\Pe) \subseteq S(\Pe)$. Furthermore, $\|Y(z)\|<1$ since $\rho_1, \rho_2$ are unital $*$-representations and $\|\mathrm{J}(z)\|_{\infty, \DC}, \|\mathrm{j}(z)\|_{\infty, \DC}<1$ for all $z \in \Pe$. We recall from \cite{Drit2007_I} the following subclass of unital $*$-representations on the $C^*$-algebra $C(K)$ over a compact set $K$.
	
	\begin{defn}\label{defn_simple}
		Given a compact set $K \subseteq \C^n$, a unital $*$-representation $\rho: C(K) \to \mathcal{B}(\HS)$ is said to be \textit{simple} if there exist $z_1, \dotsc, z_m \in K$ and orthogonal projections $P_1, \dotsc, P_m \in \mathcal{B}(\HS)$ with $P_1+\dotsc +P_m=I_\HS$ such that 
		$
		\rho(f)=f(z_1)P_1+\dotsc+f(z_m)P_m
		$
		for all $f \in C(K)$. 
	\end{defn}
In the following result, we prove that any $f \in UC(\Pe)$ is in the class $SA(\Pe)$ under the hypothesis that unital $*$-representations in a unitary colligation of $f$ are simple.	
	\begin{lem}\label{lem_prelim_III_P}
		Let $f \in UC(\Pe)$, and let $\rho_1: C(\DC) \to \mathcal{B}(\HS_1), \rho_2: C(\DC) \to \mathcal{B}(\HS_2)$ be the associated unital $*$-representations. If $\rho_1$ and $\rho_2$ are simple representations, then $f \in SA(\Pe)$. 
	\end{lem}	
	\begin{proof}
		Let $\underline{T}=(T_1, T_2, T_3) \in \mathfrak{M}_\Pe$ be defined on a Hilbert space $\mathcal{K}$. By hypothesis, 
		$f(z)=A+BY(z)(I_\HS-DY(z))^{-1}C$,
		where $\HS=\HS_1 \oplus \HS_2 \oplus \HS_3$, $V=\begin{bmatrix} A & B \\ C & D \end{bmatrix}: \C \oplus \HS \to \C \oplus \HS$ is a unitary and 
		$
		Y(z)=\begin{bmatrix}
			\rho_1(\mathrm{J}(z)) & 0 & 0 \\
			0 & \rho_2(\mathrm{j}(z)) & 0 \\
			0 & 0 & (z^{(2)}\slash 2)I_{\HS_3}
		\end{bmatrix}
		$
		for all $z=(z^{(1)}, z^{(2)}, z^{(3)}) \in \Pe$. Clearly, $f \in \text{Hol}(\Pe)$. Since $\rho_1$ is simple, there exist $\al_1, \dotsc, \al_m \in \DC$ and orthogonal projections $P_1, \dotsc, P_m \in \mathcal{B}(\HS_1)$ summing to $I_{\HS_1}$. Since $\rho_2$ is also simple, one can find $\beta_1, \dotsc, \beta_\ell \in \DC$ and orthogonal projections $Q_1, \dotsc, Q_\ell \in \mathcal{B}(\HS_2)$ whose sum equals $I_{\HS_2}$. Furthermore, $\rho_1(g)=\overset{m}{\underset{j=1}{\sum}}g(\alpha_{j})P_j$ and $\rho_2(g)=\overset{\ell}{\underset{j=1}{\sum}}g(\beta_j)Q_j$ for all $g \in C(\DC)$. Therefore, we have by \eqref{eqn_J(z)} that $\rho_1(\mathrm{J}(z))= \overset{m}{\underset{j=1}{\sum}}P_j \psi_{\al_j}(z)$ and $\rho_2(\mathrm{j}(z))=\overset{\ell}{\underset{j=1}{\sum}}Q_j \Phi_{\beta_j}(z^{(2)}, z^{(3)})$. We show that $\|Y(\underline{T})\|\leq 1$. To do so, note that
		\begin{small}
			\begin{align*}
				Y(\underline{T})
				=\begin{bmatrix}
				\rho_1(\mathrm{J}(\underline{T})) & 0 & 0 \\
					0 & \rho_2(\mathrm{j}(\underline{T})) & 0 \\					0 & 0 & I_{\HS_3}\otimes T_2\slash 2
				\end{bmatrix}
				=\begin{bmatrix}
					\overset{m}{\underset{j=1}{\sum}}P_j\otimes \psi_{\al_j}(\underline{T}) & 0 & 0 \\
					0 & \overset{\ell}{\underset{j=1}{\sum}}Q_j\otimes \Phi_{\beta_j}(T_2, T_3)  & 0 \\
					0 & 0 & I_{\HS_3}\otimes T_2\slash 2
				\end{bmatrix}.\\
			\end{align*}
		\end{small}
		For $h_1 \in \HS_1, h_2 \in \HS_2, h_3 \in \HS_3$ and $x_1, x_2, x_3 \in \mathcal{K}$, it follows that
		\begin{align*}
			&\left\|Y(\underline{T})\begin{bmatrix} h_1\otimes x_1 \\ h_2\otimes x_2 \\ h_3 \otimes x_3\end{bmatrix}\right\|^2\\
			&=\left\| \sum_{j=1}^mP_jh_1\otimes \psi_{\al_j}(\underline{T})x_1 \right \|^2+\left\|\overset{\ell}{\underset{j=1}{\sum}}Q_jh_2\otimes \Phi_{\beta_j}(T_2, T_3)x_2\right\|^2 +\|h_3 \otimes T_2x_3\slash 2\|^2\\
			&= \sum_{j=1}^m\|P_jh_1\|^2 \|\psi_{\al_j}(\underline{T})x_1\|^2+\overset{\ell}{\underset{j=1}{\sum}}\|Q_jh_2\|^2 \|\Phi_{\beta_j}(T_2, T_3)x_2\|^2 +\|h_3\|^2\|T_2x_3\slash 2\|^2\\
			& \qquad \qquad \qquad [\text{since $\{P_j: 1 \leq j \leq m\}$ and $\{Q_j: 1 \leq j \leq \ell \}$ are orthogonal projections}]\\
			& \leq \sum_{j=1}^m\|P_jh_1\|^2 \|x_1\|^2+\overset{\ell}{\underset{j=1}{\sum}}\|Q_jh_2\|^2 \|x_2\|^2 +\|h_3\|^2\|x_3\|^2 \quad [\text{as $\underline{T} \in \mathfrak{M}_\Pe$}]\\
			& = \|h_1\|^2 \|x_1\|^2+\|h_2\|^2 \|x_2\|^2 +\|h_3\|^2\|x_3\|^2 \quad \left[\text{because $\sum_{j=1}^mP_j=I_{\HS_1}$ and $\sum_{j=1}^\ell Q_j=I_{\HS_2}$}\right]\\
			&=\left\|\begin{bmatrix} h_1\otimes x_1 \\ h_2\otimes x_2 \\ h_3 \otimes x_3\end{bmatrix}\right\|^2
		\end{align*}
		and so, $Y(\underline{T})$ is a contraction. Using the representation $f(z)=A+BY(z)(I_\HS-DY(z))^{-1}C$ for all $z \in \Pe$, a simple computation shows that 
		\[
		f(\underline{T})=A\otimes I_{\mathcal{K}}+(B\otimes I_{\mathcal{K}})Y(\underline{T})\left((I_\HS\otimes I_{\mathcal{K}})-(D\otimes I_{\mathcal{K}})Y(\underline{T})\right)^{-1}(C\otimes I_{\mathcal{K}}).
		\]
		A routine calculation gives that $1-\overline{f(z)}f(z)=C^*(I-DY(z))^{-*}(I-Y(z)^*Y(z))(I-DY(z))^{-1}C$. For $R(\underline{T})=\left(I_{\HS \otimes \mathcal{K}}-(D\otimes I_{\mathcal{K}})Y(\underline{T})\right)^{-1}(C\otimes I_{\mathcal{K}})$, we can write
		\[
		I_{\mathcal{K}}-f(\underline{T})^*f(\underline{T})=R(\underline{T})^*\left(I_{\HS\otimes \mathcal{K}}-Y(\underline{T})^*Y(\underline{T})\right)R(\underline{T}).
		\]
		Since $\|Y(\underline{T})\| \leq 1$, we have $I_\mathcal{K}-f(\underline{T})^*f(\underline{T}) \geq 0$ and so, $\|f(\underline{T})\| \leq 1$. Hence, $f\in SA(\Pe)$. 
	\end{proof}
	
	We are now in a position to prove a realization theorem for functions in $SA(\Pe)$.
	
	\begin{thm}\label{thm_realization_P}
		Let $f: \Pe \to \C$ be a function. Then the following are equivalent: 
		\begin{enumerate}[leftmargin=*]
			\item[$(1)$] $f \in SA(\Pe)$;
			\item[$(2)$] $(1-f(z)\overline{f(w)})k(z, w) \succcurlyeq 0$ for all $k \in AK(\Pe)$;
			\item[$(3)$] there exist $\xi, \nabla \in C(\overline{\D})^+_\Pe$ and $\delta \in \C_\Pe^+$ such that for all $z, w \in \Pe$,
			\[
			1-f(z)\overline{f(w)}=\xi(z, w)(1-\mathrm{J}(z)\overline{\mathrm{J}(w)})+\nabla(z, w)(1-\mathrm{j}(z)\overline{\mathrm{j}(w)})+(1-\frac{z^{(2)}\overline{w}^{(2)}}{4})\delta(z, w),
			\]
			where $\mathrm{J}(z)$ and $\mathrm{j}(z)$ are as in \eqref{eqn_J(z)};
			\item[$(4)$] $f \in UC(\Pe)$.
		\end{enumerate}
	\end{thm}
	
	\begin{proof}
		The proof is organized in several steps for a better understanding for the readers.
		
		\smallskip 
		
		\noindent $(1) \implies (2)$. Let $k \in AK(\Pe)$. We first prove this implication for functions in $SA(\Pe) \cap \text{Hol}(\overline{\Pe})$. We then show that every function in $SA(\Pe)$ can be approximated by a sequence of functions from $SA(\Pe) \cap \text{Hol}(\overline{\Pe})$, which gives the desired conclusion. Let $f \in SA(\Pe) \cap \text{Hol}(\overline{\Pe})$, and let $M_z=\left(M_{z^{(1)}}, M_{z^{(2)}}, M_{z^{(3)}}\right)$ be the triple of operators of multiplication by the coordinate functions $z^{(j)}$ for $1 \leq j \leq 3$ on the reproducing kernel Hilbert space $\HS(k)$ determined by the kernel $k$. It is not difficult to see that each $M_{z^{(j)}}$ is bounded, and
		\begin{equation}\label{eqn_RP_001}
			\|M_{z^{(2)}}\|\leq 2, \quad \|\Phi_{\alpha}(M_{z^{(2)}}, M_{z^{(3)}})\| \leq 1 \quad \text{and} \quad \|\psi_{\alpha}(M_{z^{(1)}}, M_{z^{(2)}}, M_{z^{(3)}})\| \leq 1
		\end{equation}
		for all $\alpha \in \DC$. Let $\{z_1, \dotsc, z_m\} \subseteq \Pe$ and let $\HS_m(k)$ be the finite dimensional Hilbert space spanned by 
		$\{k(., z_1), \dotsc, k(., z_m)\}$. Since $M_{z^{(j)}}^*k(., z_\ell)=\overline{z}_\ell^{(j)}k(., z_\ell)$ for $1\leq \ell \leq m$ and $1 \leq j \leq 3$, it follows that $\HS_m(k)$ is jointly invariant under $\left(M_{z^{(1)}}^*, M_{z^{(2)}}^*, M_{z^{(3)}}^*\right)$. Let us define the triple $\underline{T}=(T_1, T_2, T_3)$ on $\HS_m(k)$ with each $T_j^*=M_{z^{(j)}}^*|_{\HS_m(k)}$. We have by \eqref{eqn_RP_001} that
		\[
		\|T_2^*\|\leq 2, \quad \|\Phi_{\alpha}(T_2^*, T_3^*)\| \leq 1 \quad \text{and} \quad \|\psi_{\alpha}(T_1^*, T_2^*, T_3^*)\| \leq 1
		\]
		for all $\alpha \in \DC$. Since $\G_2$ is $(1,2)$-quasi-balanced and $\Pe$ is $(1, 1, 2)$-quasi-balanced, it follows that 
		\[
		\|rT_2^*\|< 2, \quad \|\Phi_{\alpha}(rT_2^*, r^2T_3^*)\| < 1 \quad \text{and} \quad \|\psi_{\alpha}(rT_1^*, rT_2^*, r^2T_3^*)\| < 1
		\]
		for all $\alpha \in \DC$ and $0 <r<1$. Consequently, $r\cdot \underline{T}^*=(rT_1^*, rT_2^*, r^2T_3^*) \in \mathfrak{M}_\Pe$, where $\underline{T}^*=(T_1^*, T_2^*, T_3^*)$. Consider the function given by $\widehat{f}(z)=\overline{f(\bar{z})}$, which is holomorphic on $\overline{\Pe}$ as $f \in \text{Hol}(\overline{\Pe})$. Since $f \in SA(\Pe)$, we have for every $\underline{S}=(S_1, S_2, S_3) \in \mathfrak{M}_\Pe$ that $(S_1^*, S_2^*, S_3^*) \in \mathfrak{M}_\Pe$ and so, 
		\[
		\|\widehat{f}(S_1, S_2, S_3)\|=\|f(S_1^*, S_2^*, S_3^*)^*\|=\|f(S_1^*, S_2^*, S_3^*)\| \leq 1.
		\]
		Thus, $\widehat{f} \in SA(\Pe) \cap \text{Hol}(\overline{\Pe})$. Since $\overline{\Pe}$ is polynomially convex, Oka-Weil theorem guarantees the existence of a sequence of polynomials $\{p_n\}_{n \in \N}$ such that $\underset{n \to \infty}{\lim}\|p_n-f\|_{\infty, \overline{\Pe}}=0$. Clearly, the sequence of polynomials $\{\widehat{p}_n\}_{n \in \N}$ converges uniformly over $\overline{\Pe}$ to $\widehat{f}$. For each $z_j=(z_j^{(1)}, z_j^{(2)}, z_j^{(3)})$ and $r \in (0, 1)$, let $r.z_j=(rz_j^{(1)}, rz_j^{(2)}, r^2z_j^{(3)})$. Take $c_1, \dotsc, c_m \in \C$. Then 
		\begin{align*}
			\left\|\overset{m}{\underset{i, j=1}{\sum}}c_j\overline{f(r\cdot z_j)}k(., z_j)\right\|_{\HS_m(k)}^2 
			&=\lim_{n \to \infty}\left\|\overset{m}{\underset{i, j=1}{\sum}}c_j\overline{p_n(r\cdot z_j)}k(., z_j)\right\|_{\HS_m(k)}^2\\
			&=\lim_{n \to \infty}\left\|\widehat{p_n}(r.\underline{T}^*)\left(\overset{m}{\underset{i, j=1}{\sum}}c_jk(., z_j)\right)\right\|_{\HS_m(k)}^2\\
			&=\left\|\widehat{f}(r.\underline{T}^*)\left(\overset{m}{\underset{i, j=1}{\sum}}c_jk(., z_j)\right)\right\|^2_{\HS_m(k)} \\
			& \leq \left\|\overset{m}{\underset{i, j=1}{\sum}}c_jk(., z_j)\right\|^2_{\HS_m(k)}, 
		\end{align*}
		where the last inequality holds, because $\widehat{f} \in SA(\Pe)$ and $r\cdot \underline{T}^* \in \mathfrak{M}_\Pe$ for $0<r<1$. Letting $r \uparrow 1$ in the above inequality, it follows that 
		\[
		\overset{m}{\underset{i, j=1}{\sum}}\overline{c}_ic_j\left(1-f(z_i)\overline{f(z_j)}\right)k(z_i, z_j) \geq 0,
		\]
		where we have used the fact that $\langle k(., z_j), k(., z_i)\rangle_{\HS_m(k)}=k(z_i, z_j)$ for $1 \leq i, j \leq m$. Thus, $(1-f(z)\overline{f(w)})k(z, w) \succcurlyeq 0$ for all $k \in AK(\Pe)$ and $f \in SA(\Pe) \cap \text{Hol}(\overline{\Pe})$. Now, let us consider $f \in SA(\Pe)$ and let $0 <r<1$. By definition of $\Pe$, it follows that $r\cdot z \in \Pe$ for all $z \in \overline{\Pe}$. Consider the map $f_r: \overline{\Pe} \to \C$ defined as $f_r(z^{(1)}, z^{(2)}, z^{(3)})=f(rz^{(1)}, rz^{(2)}, r^2z^{(3)})$. Evidently, the map $f_r \in \text{Hol}(\overline{\Pe})$. Furthermore, $f_r \in SA(\Pe)$ since $f \in SA(\Pe)$ and $\Pe$ is a $(1, 1, 2)$-quasi-balanced domain. As proved earlier, $(1-f_r(z)\overline{f_r(w)})k(z, w) \succcurlyeq 0$ for all $k \in AK(\Pe)$ and $0<r<1$. Letting $r\uparrow 1$ gives the desired conclusion.
		
		\medskip 
		
		\noindent $(2) \implies (3)$. Suppose $(1-f(z)\overline{f(w)})k(z, w) \succcurlyeq 0$ for all $k \in AK(\Pe)$. Consider the map 
		\[
		g: \Pe \times \Pe \to \C \quad \text{defined as} \quad g(z, w)=1-f(z)\overline{f(w)}.
		\]
		Evidently, $g$ is a self-adjoint function such that $gk \succcurlyeq 0$ for all $k \in AK(\Pe)$. The desired conclusion now follows from Theorem \ref{thm_sa_P}.
		
		\medskip 
		
		\noindent $(3) \implies (4)$. Suppose there exist $\xi, \nabla \in C(\DC)^+_\Pe$ and $\delta \in \C_\Pe^+$ such that for all $z, w \in \Pe$,
		\begin{align}\label{eqn_RP_002}
			1-f(z)\overline{f(w)}=\xi(z, w)(1-\mathrm{J}(z)\overline{\mathrm{J}(w)})+\nabla(z, w)(1-\mathrm{j}(z)\overline{\mathrm{j}(w)})+(1-(z^{(2)}\overline{w}^{(2)}\slash 4))\delta(z, w).
		\end{align}
		By Proposition \ref{prop_prelim_I_P}, there exist Hilbert spaces $\HS_1, \HS_2$ and functions $L_i: \Pe \to \mathcal{B}(C(\DC),\HS_i)$ for $i=1,2$  such that $\xi(z, w)(f\overline{h})=\langle L_1(z)f, L_1(w)h \rangle_{\HS_1} $ and $\nabla(z, w)(u\overline{v})=\langle L_2(z)u, L_2(w)v \rangle_{\HS_2}$ for all $(z, w) \in \Pe \times \Pe$ and $f, h, u, v \in C(\DC)$. Also, there exist unital $*$-representations $\rho_1: C(\DC) \to \mathcal{B}(\HS_1)$ and $\rho_2: C(\DC) \to \mathcal{B}(\HS_2)$ such that 
		\begin{align}\label{eqn_RP_003}
			L_1(z)(fh)=\rho_1(f)L_1(z)(h) \quad \text{and} \quad L_2(z)(uv)=\rho_2(u)L_2(z)(v)
		\end{align}
		for all $z \in \Pe$ and $\{f, h, u, v\} \subseteq C(\DC)$. Since $\delta$ is a weak kernel on $\Pe \times \Pe$, we have by Theorem 2.53 in \cite{Agler_McCarthy} that there exist a Hilbert space $\HS_3$ and a function $g: \Pe \to \HS_3$ such that $\delta(z, w)=\langle g(z), g(w) \rangle_{\HS_3}$ for all $(z, w) \in \Pe \times \Pe$. Hence, \eqref{eqn_RP_002} can be re-written as
		\begin{align*}
			1-f(z)\overline{f(w)}
			&=\langle L_1(z)1, L_1(w)1 \rangle_{\HS_1}-\langle L_1(z)\mathrm{J}(z), L_1(w)\mathrm{J}(w) \rangle_{\HS_1}+\langle L_2(z)1, L_2(w)1 \rangle_{\HS_2}\\
			& \quad -\langle L_2(z)\mathrm{j}(z), L_2(w)\mathrm{j}(w) \rangle_{\HS_2} +\langle g(z), g(w) \rangle_{\HS_3}-\langle (z^{(2)}\slash 2)g(z), (w^{(2)}\slash 2 )g(w) \rangle_{\HS_3}
		\end{align*}
		for all $z, w \in \Pe$. Let $\HS=\HS_1\oplus \HS_2 \oplus \HS_3$. We have by \eqref{eqn_RP_003} for all $z, w \in \Pe$ that
		\begin{align}\label{eqn_RP_004}
			& 1+\left\langle Y(z) \begin{bmatrix} L_1(z)1 \\ L_2(z)1 \\ g(z)\end{bmatrix}, Y(w)\begin{bmatrix} L_1(w)1 \\ L_2(w)1 \\ g(w)\end{bmatrix} \right\rangle_{\HS} =f(z)\overline{f(w)}+\left\langle  \begin{bmatrix} L_1(z)1 \\ L_2(z)1 \\ g(z)\end{bmatrix}, \begin{bmatrix} L_1(w)1 \\ L_2(w)1 \\ g(w)\end{bmatrix} \right\rangle_{\HS}, 
		\end{align}
		where $Y(z)=\begin{bmatrix} \rho_1(\mathrm{J}(z)) & 0 & 0\\ 
			0 & \rho_2(\mathrm{j}(z)) & 0\\
			0 & 0 & (z^{(2)}\slash 2)I_{\HS_3}
		\end{bmatrix}$. Consider the subspaces of $\C \oplus \HS$ given by
		\begin{align*}
			\HS^{(1)}
			= \overline{\text{span}} \left\{
			\begin{bmatrix}
				1 \\
				Y(z)
				\begin{bmatrix}
					L_1(z)1 \\
					L_2(z)1 \\
					g(z)
				\end{bmatrix}
			\end{bmatrix}
			: z \in \Pe
			\right\} \quad \text{and} \quad \HS^{(2)}
			= \overline{\text{span}} \left\{
			\begin{bmatrix}
				f(z) \\
				\begin{bmatrix}
					L_1(z)1 \\
					L_2(z)1 \\
					g(z)
				\end{bmatrix}
			\end{bmatrix}
			: z \in \Pe
			\right\}.
		\end{align*}
		It follows from \eqref{eqn_RP_004} that the linear operator $V: \HS^{(1)} \to \HS^{(2)}$ defined as 
		\[
		V\begin{bmatrix}
			1 \\
			Y(z)
			\begin{bmatrix}
				L_1(z)1 \\
				L_2(z)1 \\
				g(z)
			\end{bmatrix}
		\end{bmatrix}
		=\begin{bmatrix}
			f(z) \\
			\begin{bmatrix}
				L_1(z)1 \\
				L_2(z)1 \\
				g(z)
			\end{bmatrix}
		\end{bmatrix}
		\]
		is an isometry. Adding an infinite-dimensional summand to $\HS^{(1)}$, if necessary, $V$ can be extended to a unitary from $\C \oplus \HS$ onto $\C \oplus \HS$ (see Section 11.3 in \cite{Agler_McCarthy}). With respect to the decomposition $\C \oplus \HS$, we can write $V=\begin{bmatrix} A & B \\ C & D \end{bmatrix}$. By definition of $V$, it then follows that 
		\begin{align}\label{eqn_RP_005}
			f(z)=A1+BY(z)\begin{bmatrix}
				L_1(z)1 \\
				L_2(z)1 \\
				g(z)
			\end{bmatrix} \quad \text{and} \quad C1+DY(z)\begin{bmatrix}
				L_1(z)1 \\
				L_2(z)1 \\
				g(z)
			\end{bmatrix}=\begin{bmatrix}
				L_1(z)1 \\
				L_2(z)1 \\
				g(z)
			\end{bmatrix}.
		\end{align}
		Since $\|Y(z)\|<1$ and $\|D\| \leq 1$, it follows from \eqref{eqn_RP_005} that $\begin{bmatrix}
			L_1(z)1 \\
			L_2(z)1 \\
			g(z)
		\end{bmatrix}=(I_\HS-DY(z))^{-1}C1$ and so, $f(z)=A+BY(z)(I_\HS-DY(z))^{-1}C$. Hence, $f \in UC(\Pe)$.  
		
		\medskip 
		
		\noindent $(4) \implies (1)$. The proof follows the same line of argument as that of Proposition 3.2 in \cite{Drit2007_I}. The similar arguments were followed in \cite{Jain} to prove the realization theorem of the tetrablock. We have already established this implication for functions in $UC(\Pe)$ associated with simple representations in Lemma \ref{lem_prelim_III_P}. We now show that an arbitrary function $f \in UC(\Pe)$ can be approximated pointwise by a net $\{f_\beta\} \subseteq UC(\Pe)$ consisting of functions associated with simple representations, from which the desired conclusion follows. We divide the proof into further steps for a clear understanding.
		\smallskip 
		
		\noindent{\textit{Step (1).}} Let $f \in UC(\Pe)$. Then there exist associated unital $*$-representations $\rho_1: C(\DC) \to \mathcal{B}(\HS_1)$ and $\rho_2: C(\DC) \to \mathcal{B}(\HS_2)$ together with a Hilbert space $\HS_3$ such that 
		\[
		f(z)=A+BY(z)(I_{\HS}-DY(z))^{-1}C, \quad \text{where} \quad Y(z)=\begin{bmatrix}
			\rho_1(\mathrm{J}(z)) & 0 & 0 \\
			0 & \rho_2(\mathrm{j}(z)) & 0 \\
			0 & 0 & (z^{(2)}\slash 2)I_{\HS_3}
		\end{bmatrix}
		\]
		and $\HS=\HS_1\oplus \HS_2 \oplus \HS_3$, and $V=\begin{bmatrix} A & B \\ C & D\end{bmatrix}$ is a unitary on $\C \oplus \HS$. By representation theorem for spectral measures, there exist unique $\mathcal{B}(\HS_1)$-valued and $\mathcal{B}(\HS_2)$-valued spectral measures $\mu_1$ and $\mu_2$ on the Borel $\sigma$-algebra of $\DC$ such that
		\begin{align}\label{eqn_RP_006}
			\rho_1(h)=\int_{\DC}h(\al)d\mu_1 \quad \text{and} \quad \rho_2(g)=\int_{\DC}g(\al)d\mu_2
		\end{align}
		for all $h, g \in C(\DC)$.
		\smallskip 
		
		\noindent{\textit{Step (2).}}  In this step, using the above representations of $\rho_1$ and $\rho_2$, we construct a net with the desired properties. Consider a collection $\mathfrak{F}$ consisting of pairs $\beta=(F, \epsilon)$, where $F$ is a finite subset of $\Pe$ and $\epsilon >0$ ordered by $(F_1, \epsilon_1)\leq (F_2, \epsilon_2)$ if $F_1 \subseteq F_2$ and $\epsilon_1 \geq \epsilon_2$. This makes $\mathfrak{F}$ a directed set. Let $\beta=(F, \epsilon) \in \mathfrak{F}$ and let us consider the collection given by
		\[
		\Lambda=\left\{\psi_{\alpha}, \chi_\alpha, \pi_2 : \al \in \DC \right\},
		\]
		where $\chi_\alpha, \pi_2: \Pe \to \C$ are given by 
		\[
		\chi_{\alpha}(z^{(1)}, z^{(2)}, z^{(3)})=\Phi_{\alpha}(z^{(2)}, z^{(3)}) \quad \text{and} \quad \pi_2(z^{(1)}, z^{(2)}, z^{(3)})=z^{(2)}\slash 2.
		\]
		Clearly, $\Lambda \subseteq B(\Pe, \DC)$ (or equivalently, $\DC^{\Pe})$, the collection of bounded functions from $\Pe$ into $\DC$, which is endowed with the topology of pointwise convergence. By Tychonov's theorem, $B(\Pe, \DC)$ is a compact Hausdorff space. Now $\DC$ is a Tychonov space in the usual metric topology (that is, points are closed and for any closed set and point
		disjoint from it, there is a continuous bounded function separating the two). Consequently, $\DC^\Pe$ is Tychonov and so, $\Lambda$ is a Tychonov space. Let $\Lambda_1=\left\{\psi_{\alpha}: \alpha \in \DC \right\}$ and $\Lambda_2=\left\{ \chi_\alpha : \alpha \in \DC \right\}$.  Consider the maps $\eta_1, \eta_2: \DC \to \Lambda$ given by
		$\eta_1(\alpha)=\psi_{\alpha}$ and $\eta_2(\alpha)=\chi_\alpha$. Evidently, $\eta_1, \eta_2$ are continuous maps with $\eta_1(\DC)=\Lambda_1$ and $\eta_2(\DC)=\Lambda_2$. Therefore, $\Lambda_1, \Lambda_2$ are compact subsets of $\Lambda$. By compactness of $\Lambda_1$, there exists a finite collection $\mathcal{U}=\{U^{\beta}_1, \dots, U^{\beta}_m\}$ of nonempty open sets in $\Lambda$ which covers $\{\psi_{\alpha} : \al \in \DC\}$ with the property that $|\psi'(z) - \psi''(z)| < \varepsilon$ for every $\psi', \psi'' \in U^{\beta}_j$ for $1 \leq j \leq m$ and $z \in F$. We construct a partition $\{\Delta_1^\beta, \dotsc, \Delta_m^\beta\}$ of $\Lambda_1$ from $\mathcal{U}$ as follows:
		\[
		\Delta^{\beta}_1 = U^{\beta}_1, \quad
		\Delta^{\beta}_2 = U^{\beta}_2 \setminus U^{\beta}_1, \quad 
		\dots, \quad 
		\Delta^{\beta}_m = U^{\beta}_m \setminus (U_1^{\beta} \cup \dotsc \cup U_{m-1}^\beta).
		\]
		Evidently, $\{\Delta^{\beta}_1, \dotsc, \Delta^{\beta}_m\}$ consists of mutually disjoint Borel subsets of $\Lambda$ that cover $\Lambda_1$ with the property that for $1 \leq j \leq m$, we have 
		\begin{align}\label{eqn_RP_007}
			|\psi'(z)-\psi''(z)|<\epsilon \quad \text{for every} \quad \psi', \psi'' \in \Delta_j^\beta, \ z \in F.
		\end{align}
		For the compact set $\Lambda_2$, we proceed similarly as above and obtain a partition $\{\Omega_1^\beta, \dotsc, \Omega_\ell^\beta\}$ of $\Lambda_2$ into Borel subsets of $\Lambda$ with the following property: for $1 \leq i \leq \ell$,
		\begin{align}\label{eqn_RP_008}
			|\chi'(z) - \chi''(z)| < \epsilon 
			\quad \text{for all } \chi', \chi'' \in \Omega_i^\beta 
			\text{ and for every } z \in F.
		\end{align}
		Also, consider collections $\{\psi_1^\beta, \dotsc, \psi_m^\beta\}$ and $\{\chi_1^\beta, \dotsc, \chi_\ell^\beta\}$ such that $\psi_j^\beta=\psi_{\theta_j} \in \Delta_j^\beta$ for some $\theta_j \in \DC$ and $\chi_i^{\beta}=\chi_{\alpha_i} \in \Omega_i^\beta$ for some $\alpha_i \in \DC$, where $j=1, \dotsc, m$ and $i=1, \dotsc, \ell$. Define the maps $\rho_{1, \beta}: C(\DC) \to \mathcal{B}(\HS_1)$ and $\rho_{2, \beta}: C(\DC) \to \mathcal{B}(\HS_2)$ as follows:
		\[
		\rho_{1, \beta}(h)=\overset{m}{\underset{j=1}{\sum}}\mu_1\left(\eta_1^{-1}(\Delta_j^\beta)\right)h(\theta_j) \quad \text{and} \quad \rho_{2, \beta}(g)=\overset{\ell}{\underset{i=1}{\sum}}\mu_2\left(\eta_2^{-1}(\Omega_i^\beta)\right)g(\al_i).
		\]
		Since $\{\Delta_1^\beta, \dotsc, \Delta_m^\beta\}$ is a partition of $\Lambda_1$ and $\eta_1(\DC)=\Lambda_1$, it follows that $\{\eta_1^{-1}(\Delta_1^\beta), \dotsc, \eta_1^{-1}(\Delta_m^\beta)\}$ is a partition of $\DC$. Consequently, the operators $\mu_1\left(\eta_1^{-1}(\Delta_j^\beta)\right)$ are pairwise orthogonal projections for $1 \leq j \leq m$ such that $\overset{m}{\underset{j=1}{\sum}}\mu_1\left(\eta_1^{-1}(\Delta_j^\beta)\right)=I_{\HS_1}$ and so, $\rho_{1, \beta}$ is a simple representation.  Similarly, one can show that $\rho_{2, \beta}$ is a simple representation. It follows from \eqref{eqn_RP_007} and \eqref{eqn_RP_008} that 
		\[
		\|\rho_{1, \beta}(\mathrm{J}(z))-\rho_1(\mathrm{J}(z))\| \leq \epsilon \quad \text{and} \quad  \|\rho_{2, \beta}(\mathrm{j}(z))-\rho_2(\mathrm{j}(z))\|\leq \epsilon
		\]
		for every $z \in F$. Consider the function $f_\beta: \Pe \to \C$ given by 
		\[
		f_\beta(z)=A+BY_\beta(z)(I_\HS-DY_\beta(z))^{-1}C, \quad \text{where} \quad Y_\beta(z)=\begin{bmatrix}
			\rho_{1, \beta}(\mathrm{J}(z)) & 0 & 0 \\
			0 & \rho_{2,\beta}(\mathrm{j}(z)) & 0 \\
			0 & 0 & (z^{(2)}\slash 2)I_{\HS_3}
		\end{bmatrix}.
		\]
		Putting everything together, we have constructed a net $\{f_\beta:\beta \in \mathfrak{F}\}$ of functions in $UC(\Pe)$ associated with simple representations. We have by Lemma \ref{lem_prelim_III_P} that each of these $f_\beta \in SA(\Pe)$.
		
		\smallskip
		
		\noindent{\textit{Step (3).}} We prove that the net $\{f_\beta\}$ converges pointwise to $f$ on $\Pe$. Let $z \in \Pe$ and $\epsilon'>0$. Choose a finite set $F$ in $\Pe$ such that $z \in F$ and consider $\beta'=(F, \epsilon') \in \mathfrak{F}$. As $\|\mathrm{J}(w)\|_{\infty, \DC}<1, \|\mathrm{j}(w)\|_{\infty, \DC}<1$ for all $w \in \Pe$, we can define positive scalars $\delta, r_1, r_2$ and $\epsilon$ as follows:
		\[
		\delta=\underset{w \in F}{\min}\left\{\frac{1-\|\mathrm{J}(w)\|_{\infty, \DC}}{2}, \frac{1-\|\mathrm{j}(w)\|_{\infty, \DC}}{2}\right\}, \ \ r_1=1-\frac{\delta}{2}, \ \ r_2=\underset{w \in F}{\max}\frac{|w^{(2)}|}{2} \ \ \text{and} \ \ \epsilon<\min\{\delta\slash 2, \epsilon'\}.
		\]
		Choose $\beta=(F, \epsilon) \in \mathfrak{F}$. By the definition of $\mathfrak{F}$, it follows that $\beta' \leq \beta$. Since $\rho_1, \rho_2$ are unital $*$-representations, we have by Step (2) that for each $w \in F$, $\|Y_\beta(w)-Y(w)\| \leq \epsilon$ together with
		\begin{align*}
			& \|\rho_{1, \beta}(\mathrm{J}(w))\| \leq \|\rho_{1, \beta}(\mathrm{J}(w))-\rho_1(\mathrm{J}(w))\|+\|\rho_1(\mathrm{J}(w))\| \leq \epsilon +1-2\delta<r_1, \ \text{and} \\
			& \|\rho_{2, \beta}(\mathrm{j}(w))\| \leq \|\rho_{2, \beta}(\mathrm{j}(w))-\rho_2(\mathrm{j}(w))\|+\|\rho_2(\mathrm{j}(w))\| \leq \epsilon +1-2\delta<r_1.
		\end{align*}
		Thus, $\|Y_\beta(w)\| =\max\{\|\rho_{1, \beta}(\mathrm{J}(w))\|, \|\rho_{2, \beta}(\mathrm{j}(w)), |w^{(2)}|\slash 2: w\in F \} \leq r$ for all $w \in F$, where $r=\max\{r_1, r_2\}$. Since $D$ is a contraction, $\|DY_\beta{(w)}\| \leq r$ for all $w \in F$. Let $w \in F$. Note that 
		\begin{small}
			\begin{align*}
				(DY_\beta(w))^n-(DY(w))^n&=D(Y_\beta(w)-Y(w))(DY_\beta(w))^{n-1}+(DY(w))D(Y_\beta(w)-Y(w))(DY_\beta(w))^{n-2}\\
				&\quad +\dotsc+(DY(w))^{n-1}D(Y_\beta(w)-Y(w)) 
			\end{align*}
		\end{small}
		and so, $\|(DY_\beta(w))^n-(DY(w))^n\| \leq n\epsilon r^{n-1}$. Then 
		\begin{align*}
			& \|Y_\beta(w)(I_\HS-DY_\beta(w))^{-1}-Y(w)(I_\HS-DY(w))^{-1}\|\\
			& \leq \|Y_\beta(w)-Y(w)\| \|(I_\HS-DY_\beta(w))^{-1}\|+\|Y(w)\| \|(I_\HS-DY_\beta(w))^{-1}-(I_\HS-DY(w))^{-1}\|\\
			& \leq \frac{\epsilon}{(1-r)^2}
		\end{align*}
		for all $w \in F$. Thus the bounded net $\{f_\beta\}$ converges pointwise to $f$ on $\Pe$. 
		
		\smallskip
		
		\noindent \textit{Step (4).} Since the net $\{f_\beta\}$ is uniformly bounded by $1$, we have by Theorem 1.4.31 in \cite{Scheidemann} that there exists a subsequence $\{f_{\beta_k}\}$ of the net that converges uniformly over compacts subsets of $\Pe$ to $f$. Let $\underline{T} \in \mathfrak{M}_\Pe$, and let $T$ be acting on a Hilbert space $\mathcal{K}$. As proved in Step (2), $f_{\beta_k} \in SA(\Pe)$ and so, $\|f_{\beta_k}(\underline{T})\| \leq 1$ for each $k$. By functional calculus presented as in \cite{Vasilescu}, it follows that 
		\[
		f_{\beta_k}(\underline{T})=\frac{1}{(2\pi i)^3}\int_{\partial \Omega} M_{\underline{T}}(z)f_{\beta_k}(z)dz,
		\]
		where $\Omega$ is an open set in $\C^3$ containing $\sigma_T(\underline{T})$ with $C^1$-boundary, $\overline{\Omega} \subseteq \Pe$ and $M_{\underline{T}}(z)$ is the Martinelli kernel corresponding to the triple $\underline{T}$. For each $x, y \in \mathcal{K}$, consider $d\mu(z)=\langle M_{\underline{T}}(z)x, y\rangle dz$, which gives a measure on the Borel subsets of $\partial \Omega$. By dominated convergence theorem, we have
		\[
		\lim_{k \to \infty}\langle f_{\beta_k}(\underline{T})x, y\rangle_{\mathcal{K}} = \frac{1}{(2\pi i)^3}\int_{\partial \Omega} \lim_{k \to \infty}f_{\beta_k}(z)d\mu(z)=\langle f(\underline{T})x, y\rangle_{\mathcal{K}} 
		\]
		and so, $\|f(\underline{T})\| \leq 1$. Thus, $f \in SA(\Pe)$, which completes the proof. 
	\end{proof}
	
	As an application of the above realization theorem for the pentablock, we present the following interpolation theorem on the pentablock.
	
	\begin{thm}\label{thm_interpolation_P}
		Let $F=\{z_1, \dotsc, z_n\} \subseteq \Pe$ and $\lm_1, \dotsc, \lm_n \in \DC$. The following are equivalent:
		\begin{enumerate}[leftmargin=*]
			\item[$(1)$] there exists a function $f \in SA(\Pe)$ such that $f(z_i)=\lm_i$ for $1 \leq i \leq n$;
			\item[$(2)$] $\begin{bmatrix} (1-\lm_i\overline{\lm_j})k(z_i, z_j)\end{bmatrix}_{i, j=1}^n \geq 0$ for all $k \in AK(\Pe)$;
			\item[$(3)$] there exist $\xi, \nabla \in C(\DC)^+_F$ and $\delta \in \C_F^+$ such that for all $i, j \in \{1, \dotsc, n\}$,
			\[
			1-\lambda_i\overline{\lambda}_j=\xi(z_i, z_j)(1-\mathrm{J}(z_i)\overline{\mathrm{J}(z_j)})+\nabla(z_i, z_j)(1-\mathrm{j}(z_i)\overline{\mathrm{j}(z_j)})+(1-(z_i^{(2)}\overline{z}_j^{(2)}\slash 4))\delta(z_i, z_j),
			\] 
			where the maps $z\mapsto \mathrm{J}(z)$ and $z\mapsto \mathrm{j}(z)$ are as in \eqref{eqn_J(z)}.	
		\end{enumerate}
	\end{thm}
	
	\begin{proof}
		The implication $(1) \implies (2)$ follows from the equivalence of $(1)$ and $(2)$ in Theorem \ref{thm_realization_P}. The part $(2) \implies (3)$ can be obtained by applying the proof of Theorem \ref{thm_sa_P} to the self-adjoint map $g : F \times F \to \mathbb{C}$ defined as $g(z_i, z_j) = 1 - \lm_i \overline{\lm}_j$. It remains to prove that $(3) \implies (1)$. Suppose there exist $\xi, \nabla \in C(\DC)^+_F$ and $\delta \in \C_F^+$ such that for all $i, j \in \{1, \dotsc, n\}$,
		\begin{align}\label{eqn_IP_P_002}
			1-\lambda_i\overline{\lambda}_j=\xi(z_i, z_j)(1-\mathrm{J}(z_i)\overline{\mathrm{J}(z_j)})+\nabla(z_i, z_j)(1-\mathrm{j}(z_i)\overline{\mathrm{j}(z_j)})+(1-(z_i^{(2)}\overline{z}_j^{(2)}\slash 4))\delta(z_i, z_j).
		\end{align} 
		By Proposition \ref{prop_prelim_I_P}, there exist Hilbert spaces $\HS_1, \HS_2$ and functions $L_1: F \to \mathcal{B}(C(\DC),\HS_1)$ and $L_2: F \to \mathcal{B}(C(\DC), \HS_2)$ satisfying 
		$\xi(z_i, z_j)(f\overline{h})=\langle L_1(z_i)f, L_1(z_j)h \rangle_{\HS_1}$ and $\nabla(z_i, z_j)(u\overline{v})=\langle L_2(z_i)u, L_2(z_j)v \rangle_{\HS_2}$ for $1 \leq i, j \leq n$ and $\{f, h, u, v\} \subseteq C(\DC)$. Also, there exist unital $*$-representations $\rho_1: C(\DC) \to \mathcal{B}(\HS_1)$ and $\rho_2: C(\DC) \to \mathcal{B}(\HS_2)$ such that 
		\begin{align}\label{eqn_IP_P_003}
			L_1(z_i)(fh)=\rho_1(f)L_1(z_i)(h) \quad \text{and} \quad L_2(z_i)(uv)=\rho_2(u)L_2(z_i)(v)
		\end{align}
		for $1 \leq i, j \leq n$ and $\{f, h, u, v\} \subseteq  C(\DC)$. Since $\delta$ is a weak kernel on $F \times F$, we have by Theorem 2.53 in \cite{Agler_McCarthy} that there exists a Hilbert space $\HS_3$ and a function $g: F \to \HS_3$ such that $\delta(z_i, z_j)=\langle g(z_i), g(z_j) \rangle_{\HS_3}$ for $1 \leq i, j \leq n$. Hence, \eqref{eqn_IP_P_002} can be re-written as
		\begin{align*}
			1-\lm_i \overline{\lm}_j
			&=\langle L_1(z_i)1, L_1(z_j)1 \rangle_{\HS_1}-\langle L_1(z_i)\mathrm{J}(z_i), L_1(z_j)\mathrm{J}(z_j) \rangle_{\HS_1}+\langle L_2(z_i)1, L_2(z_j)1 \rangle_{\HS_2} \notag \\
			& \quad -\langle L_2(z_i)\mathrm{j}(z_i), L_2(z_j)\mathrm{j}(z_j) \rangle_{\HS_2} +\langle g(z_i), g(z_j) \rangle_{\HS_3}-\langle (z_i^{(2)}\slash 2)g(z_i), (z_j^{(2)}\slash 2)g(z_j) \rangle_{\HS_3}
		\end{align*}
		for $1 \leq i, j \leq n$. Let $\HS=\HS_1\oplus \HS_2 \oplus \HS_3$. Consider the subspaces of $\C \oplus \HS$ given by
		\begin{align*}
			\HS^{(1)}
			= \overline{\text{span}} \left\{
			\begin{bmatrix}
				1 \\
				\rho_1(\mathrm{J}(z_i))L_1(z_i)1 \\
				\rho_2(\mathrm{j}(z_i))L_2(z_i)1 \\
				(z_i^{(2)}\slash 2)g(z_i)
			\end{bmatrix}
			: 1 \leq i \leq n
			\right\} \quad \text{and} \quad \HS^{(2)}
			= \overline{\text{span}} \left\{
			\begin{bmatrix}
				\lm_i \\
				L_1(z_i)1 \\
				L_2(z_i)1 \\
				g(z_i)
			\end{bmatrix}
			: 1 \leq i \leq n
			\right\}.
		\end{align*}
		Thus, the linear operator $V: \HS^{(1)} \to \HS^{(2)}$ given by
		\[
		V\begin{bmatrix}
			1 \\
			\rho_1(\mathrm{J}(z_i))L_1(z_i)1 \\
			\rho_2(\mathrm{j}(z_i))L_2(z_i)1 \\
			(z_i^{(2)}\slash 2)g(z_i)
		\end{bmatrix}
		=\begin{bmatrix}
			\lm_i \\
			L_1(z_i)1 \\
			L_2(z_i)1 \\
			g(z_i)
		\end{bmatrix}
		\]
		is an isometry. Now extend $V$ to a unitary on $\C \oplus \HS$ and so, $V=\begin{bmatrix} A & B \\ C & D \end{bmatrix} : \C \oplus \HS \to \C \oplus \HS$. We have by Part $(4) \implies (1)$ of Theorem \ref{thm_realization_P} that the map $f: \Pe \to \C$ given by
		\[
		f(z)=A+BY(z)\left(I_\HS-DY(z)\right)^{-1}C, \quad \text{where} \quad Y(z)=\begin{bmatrix} \rho_1(\mathrm{J}(z)) & 0 & 0\\ 0 & \rho_2(\mathrm{j}(z)) & 0 \\ 0 & 0 & (z^{(2)}\slash 2)I_{\HS_3} \end{bmatrix}
		\]
		is in $SA(\Pe)$. It is not difficult to see that $f(z_i)=\lambda_i$ for $1 \leq i \leq n$, which completes the proof. 
	\end{proof}

	\section{Extension theorem for the pentablock}\label{sec_ext} 
	
\noindent Let $W$ be a subset of $\Pe$. Denote by $\mathscr{HE}(W)$ the collection of all bounded functions on $W$ that admit an extension to a holomorphic function in a neighbourhood of $W$. In this section, we study the following extension problem: given a subset $W \subseteq \mathbb{P}$, determine necessary and sufficient conditions under which a non-zero function $f \in \mathscr{HE}(W)$ admits a norm-preserving extension to a function $g$ in $H^\infty(\mathbb{P})$ such that $g\slash \|f\|_{\infty, W} \in SA(\Pe)$. To do so, we introduce some terminologies and definitions. Recall from \eqref{eqn_P} that $z=(z^{(1)}, z^{(2)}, z^{(3)}) \in \Pe$ if and only if 
$|z^{(2)}|<2, |\Phi_\alpha(z^{(2)}, z^{(3)})|<1$ and $|\psi_{\al}(z^{(1)}, z^{(2)}, z^{(3)})|<1$	for all $\al \in \DC$, where 
	\[
	\Phi_\alpha(z^{(2)}, z^{(3)})=\frac{2\alpha z^{(3)}-z^{(2)}}{2-\alpha z^{(2)}} \quad \text{and} \quad  \psi_{\al}(z^{(1)}, z^{(2)}, z^{(3)})=\frac{(1-|\al|^2)z^{(1)}}{1- z^{(2)}\al+z^{(3)}\al^2}.
	\]
	Let $Q\Pe$ be the class of all commuting triples $\underline{T}=(T_1, T_2, T_3)$ of Hilbert space operators with $\sigma_T(\underline{T}) \subseteq \Pe$ such that 
	for all $\al \in \DC$,
	\[
	\|T_2\| \leq 2, \quad \|\Phi_{\al}(T_2, T_3)\| \leq 1 \quad \text{and} \quad \|\psi_{\al}(T_1, T_2, T_3)\| \leq 1,
	\]
	where $\Phi_\al(T_2, T_3)=(2\al T_3-T_2)(2-\al T_2)^{-1}$ and $\psi_\alpha(T_1, T_2, T_3)=(1-|\al|^2)T_1(I-\al T_2-\al^2T_3)^{-1}$ for $\alpha \in \DC$. Following the terminology in \cite{Mittal} for quantized domains in $\C^n$, we refer the class $Q\Pe$ as the \textit{quantum pentablock}. By quantization, we mean the process of replacing scalars with commuting operators subject to analogous norm constraints. Recall from Section \ref{sec_penta} that $\mathfrak{M}_\Pe$ is the set of all commuting triples $\underline{T}=(T_1, T_2, T_3)$ of Hilbert space operators such that 
	for all $\al \in \DC$,
	\[
	\|T_2\| <2, \quad \|\Phi_{\al}(T_2, T_3)\|<1 \quad \text{and} \quad \|\psi_{\al}(T_1, T_2, T_3)\|<1.
	\]
	It was proved in the beginning of Section \ref{sec_penta} that $\sigma_T(\underline{T}) \subseteq \Pe$ for all $\underline{T} \in \mathfrak{M}_\Pe$. Consequently, the class $\mathfrak{M}_\Pe$ is contained in the quantum pentablock $Q\Pe$.
	
	\smallskip 
	
	For a subset $W$ of $\Pe$, we say that a commuting triple $\underline{T}=(T_1, T_2, T_3)$ is \textit{subordinate} to $W$ if $\sigma_T(\underline{T}) \subset W$ and $g(\underline{T})=0$ whenever $g$ is holomorphic in a neighbourhood of $W$ and $g|_W=0$. If $f$ is a function on $W$ that admits a holomorphic extension in a neighbourhood of $W$ and $\underline{T}=(T_1, T_2, T_3)$ is subordinate to $W$, then define $f(\underline{T})$ by setting $f(\underline{T})=g(\underline{T})$, where $g$ is any holomorphic extension of $f$ in a neighbourhood of $W$. The definition of $f(\underline{T})$ is independent of the choice of the holomorphic extension $g$ of $f$. To see this, let $h$ be any other holomorphic extension of $f$ in some neighbourhood of $W$. Clearly, the function $g-h$ is holomorphic on a neighbourhood of $W$ and $g-h=0$ on $W$. Since $\underline{T}$ is subordinate to $W$, we have that $g(\underline{T})=h(\underline{T})$. Having introduced the required definitions, we present here a foretaste of the main theorem of this section. 

\begin{thm}\label{thm_ext_P}
Let $W \subseteq \Pe$ and let $f \in \mathscr{HE}(W)$ be non-zero. Then there exists $g \in H^\infty(\Pe)$ such that $\displaystyle \frac{1}{\|f\|_{\infty, W}} g \in SA(\Pe), g|_W=f$ and $\|g\|_{\infty, \Pe}=\|f\|_{\infty, W}$ if and only if 
\[
\|f(\underline{T})\| \leq \|f\|_{\infty, W} \quad \text{for every $\underline{T} \in Q\Pe$ subordinate to $W$}.
\]
\end{thm}
To prove Theorem \ref{thm_ext_P}, we adopt an approach based on the interpolation theorem, in the spirit of \cite{Agler_McCarthy_2003}, \cite{Tirtha_Sau} and \cite{Jain} for the bidisc, symmetrized bidisc and tetrablock, respectively. A first step in this direction is the following lemma which establishes the necessary condition of Theorem \ref{thm_ext_P}.

\begin{lem}
	Let $W \subseteq \Pe$ and let $f \in \mathscr{HE}(W)$ be non-zero. If there exists $g \in H^\infty(\Pe)$ such that $\displaystyle \frac{1}{\|f\|_{\infty, W}} g \in SA(\Pe), g|_W=f$ and $\|g\|_{\infty, \Pe}=\|f\|_{\infty, W}$, then $
	\|f(\underline{T})\| \leq \|f\|_{\infty, W}$ for every $\underline{T} \in Q\Pe$ subordinate to $W$.
\end{lem}	

\begin{proof}
	Let $\underline{T}=(T_1, T_2, T_3) \in Q\Pe$ be subordinate to $W$, and let $\underline{T}$ be acting on a Hilbert space $\mathcal{K}$.  Set $r\cdot z=(rz^{(1)}, rz^{(2)}, r^2z^{(3)})$ for $z=(z^{(1)}, z^{(2)}, z^{(3)}) \in \C^3$ and $0<r<1$.  Let $r_n=1-1\slash n$ for $n \in \N$. Since $\Pe$ is $(1, 1, 2)$-quasi-balanced, it follows that $r_n\cdot z \in \Pe$ for all $z \in \Pe$. Define $g_n: \Pe \to \C$ as $g_n(z)=g(r_n\cdot z)$. Then $\{g_n\}$ is a uniformly bounded sequence of holomorphic functions converging pointwise to $g$. An application of the dominated convergence theorem gives that $g_n(\underline{T})$ converges to $g(\underline{T})$ in the weak operator topology. Since $r_n\cdot \underline{T} \in \mathfrak{M}_\Pe$ and $\frac{1}{\|f\|_{\infty, W}}g \in SA(\Pe)$, we have
	\[
	|\langle g(\underline{T})x, y\rangle |=\lim_{n \to \infty}|\langle g(r_n\cdot\underline{T})x, y\rangle | \leq \lim_{n \to \infty}\|g(r_n \cdot \underline{T})\| \|x\| \| y\| \leq \|f\|_{\infty, W} \|x\| \|y\|.
	\]
Thus, $\|f(\underline{T})\| =\|g(\underline{T})\|\leq \|f\|_{\infty, W}$. 
	\end{proof}	

To prove the converse to Theorem \ref{thm_ext_P}, we first reformulate the interpolation theorem for $\Pe$, that is, Theorem \ref{thm_interpolation_P}. Let $z_1=(z_1^{(1)}, z_1^{(2)}, z_1^{(3)}), \dotsc, z_n=(z_n^{(1)}, z_n^{(2)}, z_n^{(3)})$ be distinct points in $\Pe$. We denote this data by $\textbf{z}$. Let $K_\textbf{z}$ be the collection of all $n \times n$ strictly positive definite matrices $[k(i, j)]_{i, j=1}^n$ with $k(i, i)=1$ for $1 \leq i \leq n$ such that
\begin{equation}\label{eqn_301}
	\left.
	\begin{aligned}
		&  \left[\left(1-z_i^{(2)}\overline{z}_j^{(2)}\slash 4\right)k(i, j)\right]_{i, j=1}^n \geq 0, \\[5pt]
		&  \left[\left(1-\Phi_\al(z_i^{(2)}, z_i^{(3)})\overline{\Phi_\al( z_j^{(2)}, z_j^{(3)})}\right)k(i, j)\right]_{i, j=1}^n \geq 0 
		\quad \text{for all} \ \al \in \DC, \\[5pt]
		&  \left[\left(1-\psi_{\al}(z_i)\overline{\psi_{\al}(z_j)}\right)k(i, j)\right] \geq 0 
		\quad \text{for all} \ \al \in \DC.
	\end{aligned}
	\right\}
\end{equation}
Evidently, $K_\textbf{z}$ is a subset of $n \times n$ self-adjoint complex matrices. We prove that $K_\textbf{z}$ is compact by showing that it is closed and bounded with respect to the matrix norm. The fact that $K_\textbf{z}$ is bounded follows since $\|k\| \leq \text{tr}(k)=n$ for all $k \in K_\textbf{z}$. Let $\{k_m\}$ be a sequence in $K_\textbf{z}$ that converges to $k$ in the matrix norm. By continuity arguments, it follows that $k$ is positive semi-definite and $k$ satisfies the conditions in \eqref{eqn_301}. It remains to show that $k$ is a strictly positive definite matrix. Assume on the contrary that there is a non-zero vector $v=(v_1, \dotsc, v_n)^t \in \C^n$ such that $kv=0$. Let $M_1, M_2$ and $M_3$ be  $n \times n$ diagonal matrices whose $(i, i)$-th entries are $z_i^{(1)}, z_i^{(2)}$ and $z_i^{(3)}$, respectively. It follows from \eqref{eqn_301} that $k(M_1v)=k(M_2v)=k(M_3v)=0$ and so, $k(p(M_1, M_2, M_3)v)=0$ for any holomorphic polynomial $p$ in three variables. Since $v \ne 0$, there exists some $i$ such that $v_i \ne 0$. Since $z_1, \dotsc, z_n$ are distinct, one can choose a polynomial $p$ such that $p(z_i)=1$ and $p(z_j)=0$ for $ 1 \leq j \leq n$ with $j \ne i$. Then $k(p(M_1, M_2, M_3)v)$ equals $v_i$ times the $(i, i)$-th column of $k$ and thus $k(i, i)=0$, which gives a contradiction as $k(i, i)=1$. Hence, the collection $K_\textbf{z}$ is compact. We shall use the compactness of $K_\textbf{z}$ to prove our main result. To this end, we present our next result, which is simply a reformulation of the equivalence $(1)$ and $(2)$ in Theorem \ref{thm_interpolation_P}.

\begin{thm}\label{thm_int_P_II}
Let $z_1,\dotsc, z_n$ be distinct points in $\Pe$, and let $\lm_1, \dotsc, \lm_n \in \DC$. Then there exists $g \in SA(\Pe)$ such that $g(z_i)=\lm_i$ for $1 \leq i \leq n$ if and only if $\begin{bmatrix} (1-\lm_i\overline{\lm_j})k(i, j)\end{bmatrix}_{i, j=1}^n \geq 0$ for all $k \in K_\textbf{z}$.
\end{thm} 

We now introduce an auxiliary class of functions in $H^\infty(\Pe)$. For $W \subseteq \Pe$ and $f \in \mathscr{HE}(W)$, set
\[
SA_f(\Pe)=\{g \in H^\infty(\Pe): \|g(\underline{T})\| \leq \|f\|_{\infty, W} \ \text{for all $\underline{T} \in \mathfrak{M}_\Pe$} \},
\]
which is clearly nonempty. For $g \in H^\infty(\Pe)$, let us define $\|g\|_{\mathfrak{M}_\Pe}=\sup\{\|g(\underline{T})\|: \underline{T} \in \mathfrak{M}_\Pe \}$. Evidently, $\|g\|_{\infty, \Pe} \leq \|g\|_{\mathfrak{M}_\Pe} \leq \|f\|_{\infty, W}$ for all $g \in SA_f(\Pe)$. For the given data $\textbf{z}=\{z_1, \dotsc, z_n\} \subseteq W$ and $\boldsymbol{\lambda}=(\lambda_1, \dotsc, \lambda_n) \in \C^n$, we denote by 
\[
\rho_f(\textbf{z}, \boldsymbol{\lambda})=\inf\left\{\|g\|_{\mathfrak{M}_\Pe}: g \in SA_f(\Pe) \ \text{and} \ g(z_i)=\lambda_i, 1 \leq i \leq n \right\}.
\]
It is easy to see that $|\lm_i| \leq \rho_f(\textbf{z}, \boldsymbol{\lambda})$ for $1 \leq i \leq n$. Our next result shows that $\rho_f(\textbf{z}, \boldsymbol{\lambda})$ is attained at some function $g$ in $SA_f(\Pe)$, that is, $g(z_i)=\lm_i$ for $1 \leq i \leq n$ and $\rho_f(\textbf{z}, \boldsymbol{\lambda})=\|g\|_{\mathfrak{M}_\Pe}$. Such a $g$ is called an \textit{extremal function} for the data $\textbf{z}$ and $\boldsymbol{\lambda}$.

\begin{lem}\label{lem_304}
Let $W$ be a subset of $\Pe$ and let $f \in \mathscr{HE}(W)$ be non-zero. For given $\textbf{z}=\{z_1, \dotsc, z_n\} \subseteq W$ and $\boldsymbol{\lambda}=(\lambda_1, \dotsc, \lambda_n) \in \C^n$, there exists an extremal function in $SA_f(\Pe)$ for the data $\textbf{z}$ and $\boldsymbol{\lambda}$.
\end{lem}

\begin{proof}
Let $\{g_m\}$ be a sequence in $SA_f(\Pe)$ such that $g_m(z_i)=\lm_i$ for $1 \leq i \leq n$ and $\rho_f(\textbf{z}, \boldsymbol{\lambda})=\underset{m \to \infty}\lim\|g_m\|_{\mathfrak{M}_\Pe}$. Evidently, $\|g\|_{\infty, \Pe} \leq \|g\|_{\mathfrak{M}_\Pe} \leq \|f\|_{\infty, W}$ for every $g \in SA_f(\Pe)$. Therefore, one can find a subsequence $\{g_{m_k}\}$ of $\{g_m\}$ which converges pointwise to $g \in H^\infty(\Pe)$. For every $\underline{T}$ in $\mathfrak{M}_\Pe$, an application of dominated convergence theorem gives that $g_{m_k}(\underline{T})$ converges to $g(\underline{T})$ in the weak-operator topology. Consequently, $\|g(\underline{T})\| \leq \|f\|_{\infty, W}$ for all $\underline{T} \in \mathfrak{M}_\Pe$ and $g(z_i)=\lm_i$ for $1 \leq i \leq n$. By definition of $\rho_f(\textbf{z}, \boldsymbol{\lambda})$, it follows that $\rho_f(\textbf{z}, \boldsymbol{\lambda}) \leq \|g\|_{\mathfrak{M}_\Pe}$. Let $\underline{T}$ be a commuting triple of operators acting on a Hilbert space $\mathcal{K}$, and let $\underline{T} \in \mathfrak{M}_\Pe$. Then 
for every $x, y \in \mathcal{K}$,
\[
|\langle g(\underline{T})x, y\rangle |=\lim_{k \to \infty}|\langle g_{m_k}(\underline{T})x, y\rangle | \leq \lim_{k \to \infty}\|g_{m_k}(\underline{T})\| \|x\| \| y\| \leq \lim_{k \to \infty}\|g_{m_k}\|_{\mathfrak{M}_\Pe}\|x\|\|y\|=\rho_f(\textbf{z}, \boldsymbol{\lambda}) \|x\| \|y\|.
\]
Therefore, $\|g(\underline{T})\| \leq \rho_f(\textbf{z}, \boldsymbol{\lambda})$ for all $\underline{T} \in \mathfrak{M}_\Pe$ and so, $\|g\|_{\mathfrak{M}_\Pe} \leq \rho_f(\textbf{z}, \boldsymbol{\lambda})$. The proof is complete.
\end{proof}

We now prove the following lemma, which is the final ingredient in the proof of Theorem \ref{thm_ext_P}.

\begin{lem}\label{lem_305}
	Let $W$ be a subset of $\Pe$ and let $f \in \mathscr{HE}(W)$ be non-zero. If $g \in SA_f(\Pe)$ is an extremal function for the data $\textbf{z}=\{z_1, \dotsc, z_n\} \subseteq W$ and $\boldsymbol{\lambda}=(\lambda_1, \dotsc, \lambda_n) \in \C^n$, then there exists $\underline{S} \in Q\Pe$ subordinate to $\textbf{z}$ such that $\|g(\underline{S})\|=\rho_f(\textbf{z}, \boldsymbol{\lambda})$.
\end{lem}

\begin{proof}
Suppose $g \in SA_f(\Pe)$ is an extremal function for the data $\textbf{z}=\{z_1, \dotsc, z_n\} \subseteq W$ and $\boldsymbol{\lambda}=(\lambda_1, \dotsc, \lambda_n) \in \C^n$.	For the sake of brevity, let $\rho=\rho_f(\textbf{z}, \boldsymbol{\lambda})$. If $\rho=0$, then $\|g\|_{\infty, \Pe} \leq \|g\|_{\mathfrak{M}_\Pe}=\rho=0$ and so, $g=0$. In this case, one can choose $\underline{S}=(z_1^{(1)}I, z_1^{(2)}I, z_1^{(3)}I)$ on any Hilbert space $\mathcal{K}$. Evidently, $\underline{S} \in Q\Pe$ and $\underline{S}$ is subordinate to $\textbf{z}$ with $\|g(\underline{S})\|=\rho=0$.

\smallskip 

Let us assume that $\rho>0$. Since $\|g(\underline{T})\| \leq \|g\|_{\mathfrak{M}_\Pe} = \rho$ for every $\underline{T} \in \mathfrak{M}_\Pe$, it follows that $\frac{1}{\rho}g \in SA(\Pe)$ and $\frac{1}{\rho}g(z_i)=\lm_i\slash \rho$ for $1\leq i \leq n$. We have by Theorem \ref{thm_int_P_II} that $\left[(\rho^2-\lm_i\overline{\lm}_j)k(i, j)\right]_{i, j=1}^n \geq 0$ for all $k \in K_\textbf{z}$. Consider the set 
\[
\Lambda=\left\{\lambda : 0 < \lambda \leq \|f\|_{\infty, W} \ \ \text{and} \ \ \left[(\lambda^2-\lm_i\overline{\lm}_j)k(i, j)\right]_{i, j=1}^n \geq 0 \ \text{for all} \ k \in K_\textbf{z} \right\}.
\]
Clearly, $\rho \in \Lambda$. Let $\lambda \in \Lambda$. Note that $\lambda_i \slash \lambda \in \DC$ since $(\lambda^2 - |\lambda_i|^2)k(i,i) \geq 0$ for $k \in K_{\mathbf{z}}$ and $1 \leq i \leq n$. By Theorem \ref{thm_int_P_II}, there exists $h_\lambda \in SA(\Pe)$ such that $h_\lm(z_i)=\lm_i\slash \lm$ for $1 \leq i \leq n$. Define $f_\lm=\lm h_\lm$. Then $\|f_\lm(\underline{T})\|=\lm \|h_\lm(\underline{T})\| \leq \lm \leq \|f\|_{\infty, W}$ for all $\underline{T} \in \mathfrak{M}_\Pe$ and thus, $\|f_\lm\|_{\mathfrak{M}_\Pe} \leq \lm$. Consequently, $f_\lm \in SA_f(\Pe)$ and $f_\lm(z_i)=\lm_i$ for $1 \leq i \leq n$. By definition of $\rho_f(\textbf{z}, \boldsymbol{\lambda})$, it follows that $\rho \leq \|f_\lm\|_{\mathfrak{M}_\Pe} \leq \lambda$. So, $\rho \leq \lambda$ for every $\lambda \in \Lambda$. Since $\rho \in \lm$, we have that $\rho=\inf \Lambda$. We claim that there exists a matrix $\boldsymbol{\kappa} \in K_\textbf{z}$ and a non-zero vector $y=(y_1, \dotsc, y_n)^t$ such that
\begin{equation}\label{eqn_302}
\overset{n}{\underset{i, j=1}{\sum}}(\rho^2-\lm_i\overline{\lm}_j)\boldsymbol{\kappa}(i, j)\overline{y}_iy_j=0.
\end{equation}
For $k \in K_{\textbf{z}}$, let the minimum eigenvalue of $[(\rho^2-\lm_i\overline{\lm}_j)k(i, j)]_{i, j=1}^n$ be $\mu(k)$. Let $\mu=\inf\{\mu(k): k \in K_\textbf{z}\}$. If $\mu=0$, then the claim holds trivially as $k \mapsto \mu(k)$ is a continuous map and $K_\textbf{z}$ is compact. Suppose $\mu>0$. Let $\delta=\sup\{\|k\|: k \in K_\textbf{z}\}$, which is finite as $K_\textbf{z}$ is compact. Then
\[
\left\langle \left[(\rho^2-\lm_i\overline{\lm}_j)k(i, j)\right]_{i, j=1}^n y, y \right\rangle \geq \mu(k)\|y\|^2 \geq \mu \|y\|^2
\]
for every $y \in \C^n$ and $k \in K_\textbf{z}$. For $0<\epsilon< \mu \slash \delta$, a routine computation gives that 
\[
\left[(\rho^2-\epsilon-\lm_i\overline{\lm}_j)k(i, j)\right]_{i, j=1}^n \geq 0
\]
for every $k \in K_\textbf{z}$, contradicting the fact that $\rho=\inf \Lambda$. Hence, $\mu=0$, and the above claim holds for some $\boldsymbol{\kappa} \in K_\textbf{z}$. Since $\boldsymbol{\kappa}=[\boldsymbol{\kappa}(i, j)]_{i, j=1}^n$ is strictly positive, the column space $\mathcal{K}_n=\text{span}\{\boldsymbol{\kappa}(., j) : 1\leq j \leq n \}$ is precisely $n$-dimensional. Consider the operators on $\mathcal{K}_n$ given by
\[
S_1^*\boldsymbol{\kappa}(., j)=\overline{z}_j^{(1)}\boldsymbol{\kappa}(., j), \quad S_2^*\boldsymbol{\kappa}(., j)=\overline{z}_j^{(2)}\boldsymbol{\kappa}(., j) \quad \text{and} \quad S_3^*\boldsymbol{\kappa}(., j)=\overline{z}_j^{(3)}\boldsymbol{\kappa}(., j) \qquad (1 \leq j \leq n).
\]
Evidently, $(S_1^*, S_2^*, S_3^*)$ is a commuting triple of operators with $\sigma_T(S_1^*, S_2^*, S_3^*)=\{(\overline{z}_j^{(1)}, \overline{z}_j^{(2)}, \overline{z}_j^{(3)}) : 1\leq j \leq n\}$, which implies that $\underline{S}=(S_1, S_2, S_3)$ is subordinate to $\textbf{z}$. In fact, $\underline{S}$ and $\underline{S}^*=(S_1^*, S_2^*, S_3^*)$ belong to $Q\Pe$. Also, $g(\underline{S})^*\boldsymbol{\kappa}(., j)=\overline{g(z_j)}\boldsymbol{\kappa}(., j)=\overline{\lm}_j\boldsymbol{\kappa}(., j)$ for $1 \leq j \leq n$. Since $\left[(\rho^2-\lm_i\overline{\lm}_j)\boldsymbol{\kappa}(i, j)\right] \geq 0$, it follows that $\|g(\underline{S})\|=\|g(\underline{S})^*\| \leq \rho$. The equality $\|g(\underline{S})\|=\rho$ follows directly from \eqref{eqn_302}.
\end{proof}

With these preparations, the proof of the extension theorem follows immediately.

\medskip
 
\noindent \textit{Proof of Theorem \ref{thm_ext_P}:} Let $W \subseteq \Pe$. Suppose $f \in \mathscr{HE}(W)$ is a non-zero function such that 
\begin{equation}\label{eqn_303}
\|f(\underline{T})\| \leq \|f\|_{\infty, W} \quad \text{for every $\underline{T} \in Q\Pe$ subordinate to $W$}.	
\end{equation}
 Choose a dense subset $\{z_1, z_2, \dotsc\}$ of $W$. Let $\textbf{z}_n=\{z_1, \dotsc, z_n\}$ and let $\boldsymbol{\lambda}_n=(f(z_1), \dotsc, f(z_n))$. It follows from Lemma \ref{lem_304} that there exists an extremal function $g_n \in SA_f(\Pe)$ such that $\rho_f(\textbf{z}_n, \boldsymbol{\lambda}_n)=\|g_n\|_{\mathfrak{M}_\Pe}$ for each $n$. By Lemma \ref{lem_305}, there exists $\underline{S}_n=(S_n^{(1)}, S_n^{(2)}, S_n^{(3)}) \in Q\Pe$ subordinate to $\textbf{z}_n$ such that $\|g_n(\underline{S}_n)\|=\rho_f(\textbf{z}_n, \boldsymbol{\lambda}_n)$. Consequently, 
\begin{align*}
\|g_n\|_{\mathfrak{M}_\Pe}
=\rho_f(\textbf{z}_n, \boldsymbol{\lambda}_n)
&=\|g_n(\underline{S}_n)\|  \\
&=\|f(\underline{S}_n)\| \quad [\text{since $\underline{S}_n$ is subordinate to $\textbf{z}$ and $g_n=f$ on $\textbf{z}$}]\\
&\leq \|f\|_{\infty, W} \quad [\text{$\underline{S}_n$ is subordinate to $W$ since $\textbf{z} \subseteq W$, and $f$ satisfies \eqref{eqn_303}}].
\end{align*}
Since $\|g_n\|_{\infty, \Pe} \leq \|g_n\|_{\mathfrak{M}_\Pe} \leq \|f\|_{\infty, W}$, the sequence $\{g_n\}$ is uniformly bounded. By Montel's theorem, one can find a subsequence $\{g_{n_k}\}$ of $\{g_n\}$ that converges pointwise to a function $g \in H^\infty(\Pe)$. Note that  $f(z_i)=g(z_i)$ for $i \in \N$. Therefore, $f=g$ on $W$ and so, $\|f\|_{\infty, W}=\|g\|_{\infty, W} \leq \|g\|_{\infty, \Pe}$. For every $\underline{T}$ in $\mathfrak{M}_\Pe$, an application of dominated convergence theorem gives that $g_{n_k}(\underline{T})$ converges to $g(\underline{T})$ in the weak-operator topology. Thus, $\|g(\underline{T})\| \leq \|f\|_{\infty, W}$ and so, $\|g\|_{\infty, \Pe} \leq \|f\|_{\infty, W}$. \qed

\medskip 

We say that a subset $W$ of $\Pe$ has the \textit{extension property} in $SA(\Pe)$ if for every non-zero $f \in \mathscr{HE}(W)$, there exists $g \in H^\infty(\Pe)$ such that $ g\slash \|f\|_{\infty, W} \in SA(\Pe), g|_W=f$ and $\|f\|_{\infty, W}=\|g\|_{\infty, \Pe}$.	We present below examples of subsets of $\Pe$ having the extension property. In fact, we show that the domains $\G_2$ and $\D^2$, when holomorphically embedded in $\Pe$, admit the extension property in $SA(\Pe)$. 
	
\begin{eg}\label{eg_001}
	Consider the symmetrized bidisc $\G_2=\{(z^{(1)}+z^{(2)}, z^{(1)}z^{(2)}): z^{(1)}, z^{(2)} \in \D \}$ and the subset $W=\{(0, w^{(2)}, w^{(3)}) : (w^{(2)}, w^{(3)}) \in \G_2\}$. By Theorem \ref{thm_connect_P}, $W$ is a subset of $\Pe$. Let $f \in \mathscr{HE}(W)$ be non-zero. Consider the maps $g: \Pe \to \C$ and $h: \G_2 \to \C$ given by
	$
g(w^{(1)}, w^{(2)}, w^{(3)})=f(0, w^{(2)}, w^{(3)})$ and $h(w^{(2)}, w^{(3)})=f(0, w^{(2)}, w^{(3)})$. Since $f \in \mathscr{HE}(W)$, it extends to a holomorphic function $F$ on a neighbourhood of $W$. Thus, $g(w^{(1)},w^{(2)},w^{(3)})=F(0,w^{(2)},w^{(3)})$ and $h(w^{(2)},w^{(3)})=F(0,w^{(2)},w^{(3)})$ are compositions of $F$ with holomorphic coordinate maps, and hence are holomorphic. Also, $g|_W=f$ and $\|g\|_{\infty, \Pe}=\|f\|_{\infty, W}$. Let $\underline{T}=(T_1, T_2, T_3) \in \mathfrak{M}_\Pe$. Then $(T_2, T_3) \in \mathfrak{M}_{\G_2}$ and $(T_2, T_3)$ has $\Gamma$ as a spectral set (see the discussion at the beginning of Section \ref{sec_penta}). For $r \in (0, 1)$, set $h_r(w^{(2)}, w^{(3)})=h(rw^{(2)}, r^2w^{(3)})$ for $(w^{(2)}, w^{(3)})\in \Gamma$. Clearly, $h_r$ is holomorphic in a neighbourhood of $\Gamma$. Then $
		\|g(rT_1, rT_2, r^2T_3)\|=\|h(rT_2, r^2T_3)\| =\|h_r(T_2, T_3)\|\leq \|h_r\|_{\infty, \Gamma} \leq \|h\|_{\infty, \G_2} \leq \|f\|_{\infty, W}$ 
for every $r \in (0, 1)$. Letting $r\to 1$, $\|g(\underline{T})\| \leq \|f\|_{\infty, W}$ and thus, $\frac{1}{\|f\|_{\infty, W}}g \in SA(\Pe)$. So, $W$ has the extension property in $SA(\Pe)$. \qed 
\end{eg}	

\begin{eg}\label{eg_002}
Let $W=\{(w^{(1)}, 0, w^{(3)}): w^{(1)}, w^{(3)} \in \D\}$. We have by Theorem \ref{thm_connect_P} that $W \subseteq \Pe$.  Let $f \in \mathscr{HE}(W)$ be non-zero. It follows from the definition of $\Pe$ that  $(w^{(1)}, w^{(3)}) \in \D^2$ for every $(w^{(1)}, w^{(2)}, w^{(3)}) \in \Pe$. Consider the maps $g: \Pe \to \C$ and $h: \D^2 \to \C$ given by $g(w^{(1)}, w^{(2)}, w^{(3)})=f(w^{(1)}, 0, w^{(3)})$ and $h(w^{(1)}, w^{(3)})=f(w^{(1)}, 0, w^{(3)})$. It is clear that $g$ and $h$ are holomorphic maps on $\Pe$ and $\D^2$, respectively. Also, $g|_W=f$ and $\|g\|_{\infty, \Pe}=\|f\|_{\infty, W}$. Let $\underline{T}=(T_1, T_2, T_3) \in \mathfrak{M}_\Pe$ be acting on a Hilbert space $\mathcal{K}$. Then
\begin{equation}\label{eqn_P_D2_001}
	\|T_2\|<2, \quad \|(2\al T_3-T_2)(2-\al T_2)^{-1}\|<1 \quad \text{and} \quad \|(1-|\al|^2)T_1(I-\al T_2-\al^2T_3)^{-1}\|<1 
\end{equation}
for all $\al \in \DC$. We have that $\|T_1\|<1$ by putting $\alpha=0$ in the last inequality of \eqref{eqn_P_D2_001}. The second equality in \eqref{eqn_P_D2_001} implies that $\|(2\al T_3-T_2)x\| \leq c_\al \|(2-\al T_2)x\|$ for all $\al \in \DC$ and $x \in \mathcal{K}$, where $c_\al=\|(2\al T_3-T_2)(2-\al T_2)^{-1}\|<1$. For $\alpha=\pm 1$ with $c=\max\{c_1, c_{-1}\}$ and $x \in \mathcal{K}$, we have
\begin{align*}
	2\|2T_3x\|^2+2\|T_2x\|^2
	=\|2T_3x-T_2x\|^2+\|-2T_3x-T_2x\|^2 
	& \leq c_1^2\|2x-T_2x\|^2+c_{-1}^2\|2x+T_2x\|^2\\
	& \leq 2c^2\|2x\|^2+2c^2\|T_2x\|^2 \\
	& \leq 2c^2\|2x\|^2+2\|T_2x\|^2
\end{align*}
and so, $\|T_3x\| \leq c\|x\|$. Consequently, $(T_1, T_3)$ is a commuting pair of strict contractions. For $r \in (0, 1)$, set $h_r(w^{(1)}, w^{(3)})=h(rw^{(1)}, rw^{(3)})$ for $(w^{(1)}, w^{(3)})\in \overline{\D}^2$. Clearly, $h_r$ is holomorphic in a neighbourhood of $\DC^2$. By Ando's theorem \cite{Ando}, $\|g(rT_1, rT_2, rT_3)\|=\|h(rT_1, rT_3)\| =\|h_r(T_1, T_3)\|\leq \|h_r\|_{\infty, \overline{\D}^2} \leq \|h\|_{\infty, \D^2} \leq \|f\|_{\infty, W}$ for every $r \in (0, 1)$. Letting $r\to 1$, we have that $\|g(\underline{T})\| \leq \|f\|_{\infty, W}$ and thus, $\frac{1}{\|f\|_{\infty, W}}g \in SA(\Pe)$. Consequently, $W$ has the extension property in $SA(\Pe)$. \qed 
\end{eg}

We conclude this section with the following characterization of subsets of $\Pe$ that possess the extension property in $SA(\Pe)$. The proof follows directly from Theorem \ref{thm_ext_P}.

\begin{cor}
	A subset $W$ of $\Pe$ has the extension property in $SA(\Pe)$ if and only if $\|f(\underline{T})\| \leq \|f\|_{\infty, W}$ for every $f \in \mathscr{HE}(W)$ and $\underline{T} \in Q\Pe$ subordinate to $W$. 
\end{cor}

	\section{Applications: Realization, interpolation and extension on $\D^2$ and $\G_2$}\label{sec_appl}

	\noindent In this section, we employ the realization, interpolation and extension theorems for the pentablock $\mathbb{P}$ established in Sections \ref{sec_penta} and \ref{sec_ext} to obtain the corresponding results for the bidisc $\mathbb{D}^2$ and the symmetrized bidisc $\mathbb{G}_2$. To do so, we capitalize Theorem \ref{thm_connect_P}, which states the domains $\mathbb{D}^2$ and $\mathbb{G}_2$ can be embedded into $\mathbb{P}$. This shows that the function theory on $\Pe$ naturally brings these problems together and provides a single framework from which the realization, interpolation and extension results on $\D^2$ and $\G_2$ can be derived in a unified way.

	\subsection*{The bidisc case.} Agler and McCarthy \cite{Agler_McCarthyI, Agler_McCarthy} presented a realization theorem for functions in the Schur class $S(\D^2)$, and an interpolation theorem on $\D^2$ with interpolating functions in $S(\D^2)$. The Schur-Agler class of $\D^2$ is given by
	\[
	SA(\D^2)=\{g \in \text{Hol}(\D^2) : \|g(T_1, T_2)\| \leq 1 \ \text{for all commuting pairs $(T_1, T_2)$ of strict contractions}\}.
	\]
	For $\D^2$, the classes $SA(\D^2)$ and $S(\D^2)$ coincide, e.g., see \cite{Agler_McCarthyI, Agler_McCarthy, AglerYoung2017}. Indeed, this fact follows from the Ando’s dilation theorem \cite{Ando}. With the realization and interpolation theorems on $\Pe$ in place, we revisit the corresponding realization and interpolation results for $\D^2$ proved in \cite{Agler_McCarthyI, Agler_McCarthy}. In our approach, these results are obtained through $\Pe$ and the characterizations are formulated in terms of functions belonging to $SA(\Pe)$. By Theorem \ref{thm_connect_P}, $(z^{1)}, z^{(2)}) \in \D^2$ if and only if $(z^{(1)}, 0, z^{(2)}) \in \Pe$. Also, $g \in \text{Hol}(\D^2)$ induces a holomorphic map $g\circ \theta_{\Pe \to \D^2}$ on $\Pe$, where 
	\[
	\theta_{\Pe \to \D^2}: \Pe \to \D^2, (z^{(1)}, z^{(2)}, z^{(3)}) \mapsto (z^{(1)}, z^{(3)}).
	\]
	This gives a way through which we transfer the realization and interpolation theorems on $\D^2$ to the corresponding results on $\Pe$. We present the following realization theorem for functions in $SA(\D^2)$. 
	
	\begin{thm}\label{thm_realization_P_D2}
		For a map $g : \D^2 \to \C$, the following are equivalent: 
		\begin{enumerate}
			\item[$(1)$] $g \in SA(\D^2)$;
			
			\item[$(2)$]   $f=g\circ \theta_{\Pe \to \D^2} \in SA(\Pe)$;
			
			\item[$(3)$] $(1-f(z)\overline{f(w)})k(z, w) \succcurlyeq 0$ for all $k \in AK(\Pe)$;
			\item[$(4)$] there exist $\xi, \nabla \in C(\overline{\D})^+_\Pe$ and $\delta \in \C_\Pe^+$ such that for all $z, w \in \Pe$,
			\[
			1-f(z)\overline{f(w)}=\xi(z, w)(1-\mathrm{J}(z)\overline{\mathrm{J}(w)})+\nabla(z, w)(1-\mathrm{j}(z)\overline{\mathrm{j}(w)})+(1-\frac{z^{(2)}\overline{w}^{(2)}}{4})\delta(z, w),
			\]
			where $\mathrm{J}(z)$ and $\mathrm{j}(z)$ are as in \eqref{eqn_J(z)};
			\item[$(5)$] $f \in UC(\Pe)$.
		\end{enumerate}
	\end{thm}
	
	\begin{proof} The equivalence of $(2)$ with $(3)$–$(5)$ follows from Theorem \ref{thm_realization_P}. We prove $(1)\iff(2)$. 
		
		\medskip 
		
		\noindent $(1) \implies (2)$. Let $g \in SA(\D^2)$. Take $\underline{T}=(T_1, T_2, T_3) \in \mathfrak{M}_\Pe$ acting on a Hilbert space $\mathcal{K}$. It follows from Example \ref{eg_002} that $T_1, T_3$ are commuting strict contractions. Consequently, $(T_1, T_3) \in SA(\D^2)$ and $\|g\circ \theta_{\Pe \to \D^2}(\underline{T})\|=\|g(T_1, T_3)\| \leq 1$. Therefore, $g\circ \theta_{\Pe \to \D^2} \in SA(\Pe)$. 
		
		\medskip 
		
		\noindent $(2) \implies (1)$. Suppose $f=g\circ \theta_{\Pe \to \D^2} \in SA(\Pe)$. It follows from Theorem \ref{thm_connect_P} that $g$ is holomorphic on $\D^2$ since $g(z^{(1)}, z^{(3)})=f(z^{(1)}, 0, z^{(3)})$ for all $(z^{(1)}, z^{(3)}) \in \D^2$. Let $(T_1, T_3)$ be a commuting pair of strict contractions. Define $\underline{T}=(T_1, T_2, T_3)$ with $T_2=0$. We show that $\underline{T} \in \mathfrak{M}_\Pe$. Clearly, $\|T_2\|<2$ and $\|\Phi_\al(T_1, 0, T_3)\|=\|\al T_3\|<1$ for all $\al \in \DC$. Fix $\al \in \DC$. Then
		\begin{align}\label{eqn_P_D2_002}
			\|\psi_\al(T_1, 0, T_3)\|
			=(1-|\al|^2)\|T_1(I-\al^2T_3)^{-1}\|
			\leq (1-|\al|^2)\|T_1\| \overset{\infty}{\underset{m=0}{\sum}}\|\al^2 T_3\|^m
			=\|T_1\|\frac{1-|\al|^2}{1-|\al|^2\|T_3\|}.
		\end{align}
		It was proved in Proposition 4.2 of \cite{AglerIV} that for $(z^{(2)}, z^{(3)}) \in \G_2$, 
		\begin{equation}\label{eqn_P_D2_003}
			\sup_{\lm \in \DC}\frac{1-|\lm|^2}{|1-z^{(2)}\lm+z^{(3)}\lm^2|}=\frac{1-|\lm_0|^2}{|1-z^{(2)}\lm_0+z^{(3)}\lm_0^2|},
		\end{equation}
		where $\lm_0=\overline{\beta}(1+\sqrt{1-|\beta|^2})^{-1}$ and $\beta=(z^{(2)}-\overline{z}^{(2)} z^{(3)})(1-|z^{(3)}|^2)^{-1}$. In particular, choose $\lm=i|\al|$ and $(z^{(2)}, z^{(3)})=(0, \|T_3\|) \in \G_2$. For this choice of $(z^{(2)}, z^{(3)})$, we have that $\lambda_0=0$. Then by \eqref{eqn_P_D2_002} and \eqref{eqn_P_D2_003}, it follows that
		$\displaystyle 
		\|\psi_\al(T_1, 0, T_3)\| \leq \|T_1\| \frac{1-|\al|^2}{1-|\al|^2\|T_3\|} \leq \|T_1\| < 1 
		$
		and so, $\underline{T} \in \mathfrak{M}_\Pe$. Thus, $\|g(T_1, T_3)\|=\|g\circ \theta_{\Pe \to \D^2}(T_1, 0, T_3)\| \leq 1$ and $g \in SA(\D^2)$.
	\end{proof}
	
	As an application of Theorems \ref{thm_interpolation_P} and \ref{thm_realization_P_D2}, we have the following interpolation theorem on $\D^2$.
	
	\begin{thm}\label{thm_interpolation_P_D2}
		Let $F=\{z_1, \dotsc, z_n\} \subseteq \D^2$ and let $\lm_1, \dotsc, \lm_n \in \DC$. Then the following are equivalent:
		\begin{enumerate}[leftmargin=*]
			\item[$(1)$] there exists $g \in SA(\D^2)$ such that $g(z_j)=\lm_j$ for $1 \leq j \leq n$;
			
			\item[$(2)$] there exists $f \in SA(\Pe)$ such that $f(z_j^{(1)}, 0, z_j^{(2)})=\lm_j$ for $z_j=(z_j^{(1)}, z_j^{(2)})$ with $1 \leq j \leq n$; \smallskip 
			
			\item[$(3)$] $\begin{bmatrix} (1-\lm_i\overline{\lm_j})k\left((z_i^{(1)}, 0, z_i^{(2)}), (z_j^{(1)}, 0, z_j^{(2)})\right)\end{bmatrix}_{i, j=1}^n \geq 0$ for all $k \in AK(\Pe)$; \smallskip 
			
			\item[$(4)$] there exist $\xi, \nabla \in C(\DC)^+_F$ such that for $w_i=(z_i^{(1)}, 0, z_i^{(2)})$ and $1 \leq i, j \leq  n$, we have
			\[
			1-\lambda_i\overline{\lambda}_j=\xi(w_i, w_j)(1-\mathrm{J}(w_i)\overline{\mathrm{J}(w_j)})+\nabla(w_i, w_j)(1-\mathrm{j}(w_i)\overline{\mathrm{j}(w_j)}),
			\]
	where the maps $z\mapsto \mathrm{J}(z)$ and $z\mapsto \mathrm{j}(z)$ are as in \eqref{eqn_J(z)}.
		\end{enumerate}
In the case when $(1)$ holds, one can choose $f=g\circ \theta_{\Pe \to \D^2}$ in $(2)$.
	\end{thm}
	
	\begin{proof}
		The part $(1) \implies (2)$ follows directly from Theorem \ref{thm_realization_P_D2} by choosing $f=g\circ \theta_{\Pe \to \D^2}$. Moreover, the implications $(2) \implies (3) \implies (4) \implies (2)$ follow from Theorem \ref{thm_interpolation_P}. We now prove $(2) \implies (1)$. Suppose there exists $f \in SA(\Pe)$ such that $f(z_j^{(1)}, 0, z_j^{(2)})=\lm_j$ for $1 \leq j \leq n$. Define $g: \D^2 \to \C$ as $g(z^{(1)}, z^{(2)})=f(z^{(1)}, 0, z^{(2)})$. By Theorem \ref{thm_connect_P}, $g \in \text{Hol}(\D^2)$. For commuting pair $(T_1, T_3)$ of strict contractions, it follows from the proof of Theorem \ref{thm_realization_P_D2} that  $(T_1, 0, T_3) \in \mathfrak{M}_\Pe$ and so, $\|g(T_1, T_3)\|=\|f(T_1, 0, T_3)\| \leq 1$. Thus, $g \in SA(\D^2)$ with each $g(z_j)=\lambda_j$.
	\end{proof}

	Let $V$ be a subset of $\D^2$. A commuting pair $(A, B)$ of contractions is said to be \textit{subordinate} to $V$ if $\sigma_T(A, B) \subset V$, and $g(A, B)=0$ whenever $g$ is holomorphic in a neighbourhood of $V$ and $g|_V=0$. If $f$ is a function on $V$ having a holomorphic extension in a neighbourhood of $V$ and $(A, B)$ is subordinate to $V$, then define $f(A, B)=g(A, B)$, where $g$ is any holomorphic extension of $f$ in a neighbourhood of $W$. The following extension theorem on $\D^2$ was proved by Agler and McCarthy in \cite{Agler_McCarthy_2003}. We provide an alternative proof here based on Theorem \ref{thm_ext_P}.

\begin{thm}
	Let $V \subseteq \D^2$ and let $h \in \mathscr{HE}(V)$. Then there is a bounded holomorphic function $\widetilde{h}$ on $\D^2$ such that $\widetilde{h}|_V=h$ and $\|h\|_{\infty, V}=\|\widetilde{h}\|_{\infty, \D^2}$ if and only if $\|h(A, B)\| \leq \|h\|_{\infty, V}$ for every commuting pair of contractions $(A, B)$ subordinate to $V$.
\end{thm}

\begin{proof} 
Set $W=\{(z^{(1)}, 0, z^{(3)}): (z^{(1)}, z^{(3)}) \in V\}$ and $W_0=\{(z^{(1)}, 0, z^{(3)}): (z^{(1)}, z^{(3)}) \in \D^2\}$. We have by Theorem \ref{thm_connect_P} that $W \subseteq W_0 \subseteq \Pe$. Define $f: W \to \C$ as $f(z^{(1)}, 0, z^{(3)})=h(z^{(1)}, z^{(3)})$.  Since $h \in \mathscr{HE}(V)$, it extends to a holomorphic map $h_0$ on a neighbourhood $U$ of $V$. Then $F(z^{(1)},z^{(2)},z^{(3)}) = h_0(z^{(1)},z^{(3)})$ is a holomorphic extension of $f$ to the open neighbourhood $\{(z^{(1)}, z^{(2)}, z^{(3)}) \in \C^3: (z^{(1)}, z^{(3)}) \in U\}$ of $W$ such that $F|_W=f$ and so, $f \in \mathscr{HE}(W)$. 
	
	\smallskip 
	
	$(\implies)$ Suppose there is a bounded holomorphic function $\widetilde{h}$ on $\D^2$ such that $\widetilde{h}|_V=h$ and $\|h\|_{\infty, V}=\|\widetilde{h}\|_{\infty, \D^2}$. The map $\widetilde{f}: W_0 \to \C$ given by $\widetilde{f}(z^{(1)}, 0, z^{(3)})=\widetilde{h}(z^{(1)}, z^{(3)})$ has a holomorphic extension on $\Pe$, because $\widetilde{F}(z^{(1)}, z^{(2)}, z^{(3)})= \widetilde{h}(z^{(1)}, z^{(3)})$ is holomorphic on $\Pe$ that satisfies $\widetilde{F}|_{W_0}=\widetilde{f}$. Thus, $\widetilde{f} \in \mathscr{HE}(W_0)$. By Example \ref{eg_002}, $W_0$ has the extension property in $SA(\Pe)$ and by Theorem \ref{thm_ext_P}, 
	\begin{equation}\label{eqn_4003}
		\|\widetilde{f}(\underline{T})\| \leq \|\widetilde{f}\|_{\infty, W_0}\leq \|\widetilde{h}\|_{\infty, \D^2}=\|h\|_{\infty, V}  
	\end{equation}
	for every $\underline{T} \in Q\Pe$ subordinate to $W_0$. Suppose $(T_1, T_3)$ is a commuting pair of contractions subordinate to $V$. Then $\sigma_T(T_1,0, T_3) = \{(z^{(1)}, 0, z^{(3)}) : (z^{(1)}, z^{(3)}) \in \sigma_T(T_1, T_3)\} \subseteq \{(z^{(1)}, 0, z^{(3)}) : (z^{(1)}, z^{(3)}) \in V\} = W$. Following the proof of $(2) \implies (1)$ in Theorem \ref{thm_realization_P_D2}, one can easily show that $(T_1, 0, T_3) \in Q\Pe$. Let $G$ be a holomorphic map on a neighbourhood $U \subseteq \mathbb{C}^3$ of $W$ such that $G|_W = 0$. Define $\xi : \mathbb{C}^2 \to \mathbb{C}^3$ as $\xi(z^{(1)}, z^{(3)}) = (z^{(1)}, 0, z^{(3)})$ and set
	$g(z^{(1)}, z^{(3)})= G(z^{(1)}, 0, z^{(3)})$ for $(z^{(1)}, z^{(3)}) \in \xi^{-1}(U)$.
	Clearly, $g$ is holomorphic in a neighbourhood of $V$. For $(z^{(1)}, z^{(3)}) \in V$, we have that $g(z^{(1)}, z^{(3)}) = G(z^{(1)}, 0, z^{(3)}) = 0$ as $G|_W=0$. Since $(T_1, T_3)$ is subordinate to $V$ and $g|_V=0$, it follows that
	$g(T_1, T_3) = 0$. By holomorphic functional calculus, $G(T_1, 0, T_3) = g(T_1, T_3)= 0$. Thus, $(T_1, 0, T_3)$ is subordinate to $W$ and so, $(T_1, 0, T_3)$ is subordinate to $W_0$ since $W \subseteq W_0$. By \eqref{eqn_4003},
	$
	\|h(T_1, T_3)\|=\|\widetilde{h}(T_1, T_3)\|=\|\widetilde{F}(T_1, 0, T_3)\|=\|\widetilde{f}(T_1, 0, T_3)\| \leq \|h\|_{\infty, V}
	$
	for every commuting pair of contractions $(T_1, T_3)$ subordinate to $V$.
	
	\smallskip 
	
	$(\impliedby)$ Let $\|h(T_1, T_3)\| \le \|h\|_{\infty,V}$ for every commuting pair of contractions $(T_1, T_3)$ subordinate to $V$. Let $\underline{T} = (T_1,T_2,T_3) \in Q\Pe$ be subordinate to $W$. Following similar computations as in Example \ref{eg_002}, we have that $\|T_1\|, \|T_3\| \leq 1$. Also, $\sigma_T(\underline{T}) \subseteq W=\{(z^{(1)}, 0, z^{(3)}): (z^{(1)}, z^{(3)}) \in V\}$.
	Let $G(z^{(1)},z^{(2)},z^{(3)}) = z^{(2)}$. Clearly, $G$ is holomorphic on $\C^3$ and $G|_W = 0$. Since $\underline{T}$ is subordinate to $W$, it follows that $T_2 = G(T_1,T_2,T_3) = 0$ and so, $\underline{T} = (T_1, 0,T_3)$. We show that $(T_1, T_3)$ is subordinate to $V$. By spectral mapping principle, $\sigma_T(T_1, T_3) \subseteq V$. Let $g_0$ be a holomorphic function in a neighbourhood $U$ of $V$ with $g_0|_V=0$. Then the function $G_0(z^{(1)},z^{(2)},z^{(3)})= g_0(z^{(1)},z^{(3)})$ is holomorphic on $\{(z^{(1)}, z^{(2)}, z^{(3)}) \in \C^3: (z^{(1)}, z^{(3)}) \in U\}$, which is a neighbourhood of $W$ and $G_0|_W=0$. Since $\underline{T}$ is subordinate to $W$, we have that $g_0(T_1,T_3) = G_0(T_1, 0, T_3) = G_0(T_1,T_2,T_3) = 0$.
	Hence, $(T_1,T_3)$ is subordinate to $V$. By hypothesis,
	$
	\|f(\underline{T})\| = \|f(T_1, 0, T_3)\| = \|h(T_1,T_3)\| \le \|h\|_{\infty,V} = \|f\|_{\infty,W}.
	$
	By Theorem \ref{thm_ext_P}, there exists $\widehat{F} \in H^\infty(\Pe)$ such that $\widehat{F}|_W = f$ and $\|\widehat{F}\|_{\infty,\Pe} = \|f\|_{\infty,W}$. Define $\widehat{h} : \D^2 \to \C$ by $\widehat{h}(z^{(1)},z^{(3)})= \widehat{F}(z^{(1)}, 0, z^{(3)})$, which is holomorphic on $\D^2$. For $(z^{(1)},z^{(3)}) \in V$, we have that $\widehat{h}(z^{(1)},z^{(3)}) = \widehat{F}(z^{(1)}, 0, z^{(3)}) = f(z^{(1)}, 0, z^{(3)})= h(z^{(1)},z^{(3)})$ and thus, $\widehat{h}|_V = h$. Finally,
	$\|\widehat{h}\|_{\infty,\D^2} \le \|\widehat{F}\|_{\infty,\Pe} = \|f\|_{\infty,W} = \|h\|_{\infty,V}$ and so, $\|\widehat{h}\|_{\infty,\D^2} = \|h\|_{\infty,V}$, which completes the proof.
\end{proof}

\subsection{The symmetrized bidisc case.} We now present the realization and interpolation results on the symmetrized bidisc through the pentablock framework. The authors of \cite{AglerYoung2017, Tirtha_Sau} established a realization theorem for the Schur class $S(\G_2)$, and an interpolation theorem on $\G_2$ with interpolating functions belonging to $S(\G_2)$. As mentioned earlier, the Schur class $S(\G_2)$ and Schur-Agler class $SA(\G_2)$ coincides with each other. Building on the realization and interpolation theorems on $\Pe$, we obtain realization and interpolation theorems on $\G_2$ with interpolating functions belonging to $SA(\G_2)$. Our characterizations are formulated in terms of functions in $SA(\Pe)$, thereby providing an alternative criterion for these results on $\G_2$ to those obtained in \cite{AglerYoung2017, Tirtha_Sau}. To do so, we recall from Theorem \ref{thm_connect_P} that $(z^{1)}, z^{(2)}) \in \G_2$ if and only if $(0, z^{(1)}, z^{(2)}) \in \Pe$. Also, a holomorphic map $g: \G_2 \to \C$ induces a holomorphic map $g\circ \theta_{\Pe \to \G_2}$ on $\Pe$, where 
	\[
	\theta_{\Pe \to \G_2}: \Pe \to \G_2, (z^{(1)}, z^{(2)}, z^{(3)}) \mapsto (z^{(2)}, z^{(3)}).
	\]
	To begin with, we present the realization theorem for functions in $SA(\G_2)$.
	
	\begin{thm}\label{thm_realization_P_G2}
		For a function $g : \G_2 \to \C$, the following are equivalent:
		
		\begin{enumerate}
			\item[$(1)$] $g \in SA(\G_2)$;
			\item[$(2)$]   $f=g\circ \theta_{\Pe \to \G_2} \in SA(\Pe)$;
			
			\item[$(3)$] $(1-f(z)\overline{f(w)})k(z, w) \succcurlyeq 0$ for all $k \in AK(\Pe)$;
			\item[$(4)$] there exist $\xi, \nabla \in C(\overline{\D})^+_\Pe$ and $\delta \in \C_\Pe^+$ such that for all $z, w \in \Pe$,
			\[
			1-f(z)\overline{f(w)}=\xi(z, w)(1-\mathrm{J}(z)\overline{\mathrm{J}(w)})+\nabla(z, w)(1-\mathrm{j}(z)\overline{\mathrm{j}(w)})+(1-\frac{z^{(2)}\overline{w}^{(2)}}{4})\delta(z, w),
			\]
	where the maps $z\mapsto \mathrm{J}(z)$ and $z\mapsto \mathrm{j}(z)$ are as in \eqref{eqn_J(z)};
			\item[$(5)$] $f \in UC(\Pe)$.
		\end{enumerate}
	\end{thm}
	
	\begin{proof} In view of Theorem \ref{thm_realization_P}, it suffices to prove $(1) \iff (2)$.	Let $g \in SA(\G_2)$ and let $\underline{T}=(T_1, T_2, T_3) \in \mathfrak{M}_\Pe$. Evidently, $(T_2, T_3) \in \mathfrak{M}_{\G_2}$ and so, $\|g\circ \theta_{\Pe \to \G_2}(\underline{T})\|=\|g(T_2, T_3)\| \leq 1$. Therefore, $g\circ \theta_{\Pe \to \G_2} \in SA(\Pe)$. Conversely, suppose $g\circ \theta_{\Pe \to \G_2} \in SA(\Pe)$. We have by Theorem \ref{thm_connect_P} that $g \in \text{Hol}(\G_2)$ since $g(z^{(2)}, z^{(3)})=f(0, z^{(2)}, z^{(3)})$ for all $(z^{(2)}, z^{(3)}) \in \G_2$.  Take $(T_2, T_3) \in \mathfrak{M}_{\G_2}$. Clearly, $(0, T_2, T_3) \in \mathfrak{M}_\Pe$. Thus, $\|g(T_2, T_3)\|=\|g\circ \theta_{\Pe \to \G_2}(0, T_2, T_3)\| \leq 1$ and so, $g \in SA(\G_2)$.
	\end{proof}
	
	An application of Theorems \ref{thm_interpolation_P} and \ref{thm_realization_P_G2} provides the following interpolation theorem on $\G_2$.
	
	\begin{thm}\label{thm_interpolation_P_G2}
		Let $F=\{z_1, \dotsc, z_n\} \subseteq \G_2$ and let $\lm_1, \dotsc, \lm_n \in \DC$. Then the following are equivalent:
		\begin{enumerate}[leftmargin=*]
			\item[$(1)$] there exists $g \in SA(\G_2)$ such that $g(z_j)=\lm_j$ for $1 \leq j \leq n$;
			
			\item[$(2)$] there exists $f \in SA(\Pe)$ such that $f(0, z_j)=\lm_j$ for $1 \leq j \leq n$; \smallskip 
			
			\item[$(3)$] $\begin{bmatrix} (1-\lm_i\overline{\lm_j})k\left((0, z_i), (0, z_j)\right)\end{bmatrix}_{i, j=1}^n \geq 0$ for all $k \in AK(\Pe)$; \smallskip 
			
			\item[$(4)$] there exist $\xi, \nabla \in C(\DC)^+_F$ and $\delta \in \C_F^+$ such that for $w_i=(0, z_i)$ and $1 \leq i, j \leq  n$, we have
			\begin{small}
				\[
				1-\lambda_i\overline{\lambda}_j=\xi(w_i, w_j)(1-\mathrm{J}(w_i)\overline{\mathrm{J}(w_j)})+\nabla(w_i, w_j)(1-\mathrm{j}(w_i)\overline{\mathrm{j}(w_j)})+(1-(w_i^{(2)}\overline{w}_j^{(2)}\slash 4))\delta(w_i, w_j),
				\]
				where the maps $z \mapsto \mathrm{J}(z)$ and $z\mapsto \mathrm{j}(z)$ are as in \eqref{eqn_J(z)}.	
			\end{small}
		\end{enumerate}
		In the case when $(1)$ holds, one may choose $f=g\circ \theta_{\Pe \to \G_2}$ in $(2)$.
	\end{thm}
	
	\begin{proof}
		The part $(1) \implies (2)$ follows directly from Theorem \ref{thm_realization_P_G2} by choosing $f=g\circ \theta_{\Pe \to \G_2}$. We show that $(2) \implies (1)$. Suppose there exists $f \in SA(\Pe)$ such that $f(0, z_j)=\lm_j$ for $1 \leq j \leq n$. Define $g: \G_2 \to \C$ as $g(z^{(2)}, z^{(3)})=f(0, z^{(2)}, z^{(3)})$. By Theorem \ref{thm_connect_P}, $g \in \text{Hol}(\G_2)$. For $(T_2, T_3) \in \mathfrak{M}_{\G_2}$,  $(0, T_2, T_3) \in \mathfrak{M}_\Pe$ and so, $\|g(T_2, T_3)\|=\|f(0, T_2, T_3)\| \leq 1$. Thus, $g \in SA(\G_2)$ and $g(z_j)=\lambda_j$ for $ 1\leq j \leq n$. The implications $(2) \implies (3) \implies (4) \implies (2)$ follow from Theorem \ref{thm_interpolation_P}.
	\end{proof}

A commuting pair of operators $(S, P)$ is said to be a \textit{$\Gamma$-contraction} if $\|p(S, P)\| \leq \|p\|_{\infty, \Gamma}$ for all holomorphic polynomials in two variables. Here, $\Gamma=\overline{\G}_2$. See \cite{Tirtha_Pal_Roy} for a detailed study of $\Gamma$-contractions. Let $V$ be a subset of $\G_2$. A $\Gamma$-contraction $(S, P)$ is said to be \textit{subordinate} to $V$, if $\sigma_T(S, P) \subset V$ and $g(S, P)=0$ whenever $g$ is holomorphic in a neighbourhood of $V$ and $g|_V=0$. If $f$ is a function on $V$ which admits a holomorphic extension in a neighbourhood of $V$ and $(S, P)$ is subordinate to $V$, then we set $f(S, P)=g(S, P)$, where $g$ is any holomorphic extension of $f$ in a neighbourhood of $W$. The authors of \cite{Tirtha_Sau, Agler2019} established the following extension theorem on $\G_2$. Below, we provide an alternative proof here capitalizing on Theorem \ref{thm_ext_P}.	

\begin{thm}
	Let $V \subseteq \G_2$ and let $h \in \mathscr{HE}(V)$. Then there is a bounded holomorphic function $\widetilde{h}$ on $\G_2$ such that $\widetilde{h}|_V=h$ and $\|h\|_{\infty, V}=\|\widetilde{h}\|_{\infty, \G_2}$ if and only if $\|h(S, P)\| \leq \|h\|_{\infty, V}$ for every $\Gamma$-contraction $(S, P)$ subordinate to $V$.
\end{thm}

\begin{proof} 
	Set $W = \{0\} \times V$ and $W_0=\{0\} \times \G_2$. Clearly, $W \subseteq W_0 \subseteq \Pe$. Define $f: W \to \C$ as $f(0, z^{(2)}, z^{(3)})=h(z^{(2)}, z^{(3)})$ for all $(z^{(2)}, z^{(3)}) \in V$.  Since $h \in \mathscr{HE}(V)$, it extends to a holomorphic map $h_0$ on a neighbourhood $U$ of $V$. Then the map $F(z^{(1)},z^{(2)},z^{(3)}) = h_0(z^{(2)},z^{(3)})$ is a holomorphic extension of $f$ to $\mathbb{C} \times U$ such that $F|_W=f$ and so, $f \in \mathscr{HE}(W)$. 
	
	\smallskip 
	
$(\implies)$ Suppose there is a bounded map $\widetilde{h} \in \text{Hol}(\G_2)$ such that $\widetilde{h}|_V=h$ and $\|h\|_{\infty, V}=\|\widetilde{h}\|_{\infty, \G_2}$. The map $\widetilde{f}: W_0 \to \C$ given by $\widetilde{f}(0, z^{(2)}, z^{(3)})=\widetilde{h}(z^{(2)}, z^{(3)})$ has a holomorphic extension on $\Pe$, because $\widetilde{F}(z^{(1)}, z^{(2)}, z^{(3)})= \widetilde{h}(z^{(2)}, z^{(3)})$ is holomorphic on $\Pe$ satisfying $\widetilde{F}|_{W_0}=\widetilde{f}$. So, $\widetilde{f} \in \mathscr{HE}(W_0)$. By Example \ref{eg_001}, $W_0=\{0\} \times \G_2$ has the extension property in $SA(\Pe)$ and by Theorem \ref{thm_ext_P}, 
	\begin{equation}\label{eqn_4004}
	\|\widetilde{f}(\underline{T})\| \leq \|\widetilde{f}\|_{\infty, W_0}\leq \|\widetilde{h}\|_{\infty, \G_2}=\|h\|_{\infty, V}  
	\end{equation}
	for every $\underline{T} \in Q\Pe$ subordinate to $W_0$. Suppose $(S,P)$ is a $\Gamma$-contraction subordinate to $V$. Note that $\sigma_T(0,S,P) = \{0\} \times \sigma_T(S,P) \subseteq \{0\} \times V = W \subseteq \Pe$. It now follows from Theorem 1.5 in \cite{AglerYoung2003} that $\|P\| \leq 2$ and $\|\Phi_\al(S, P)\| \leq 1$ for all $\alpha \in \DC$. Consequently, $(0, S, P) \in Q\Pe$. Let $G$ be a holomorphic map in a neighbourhood $U \subseteq \mathbb{C}^3$ of $W$ such that $G|_W = 0$. Define $\xi : \mathbb{C}^2 \to \mathbb{C}^3$ as $\xi(z^{(2)}, z^{(3)}) = (0, z^{(2)}, z^{(3)})$ and set $g(z^{(2)}, z^{(3)}) = G(0, z^{(2)}, z^{(3)})$ for $(z^{(2)}, z^{(3)}) \in \xi^{-1}(U)$. Clearly, $g$ is holomorphic in a neighbourhood of $V$. For $(z^{(2)}, z^{(3)}) \in V$, we have that $g(z^{(2)}, z^{(3)}) = G(0, z^{(2)}, z^{(3)}) = 0$ as $G|_W=0$. Since $(S,P)$ is subordinate to $V$ and $g|_V=0$, it follows that
	$g(S,P) = 0$. By functional calculus, $G(0,S,P) = g(S,P) = 0$.
	Thus, $(0,S,P)$ is subordinate to $W$. Since $W \subseteq W_0$, it follows that $(0,S,P)$ is subordinate to $W_0$.  By \eqref{eqn_4004}, $
	\|h(S, P)\|=\|\widetilde{h}(S, P)\|=\|\widetilde{F}(0, S, P)\|=\|\widetilde{f}(0, S, P)\| \leq \|h\|_{\infty, V}
	$
	for every $\Gamma$-contraction $(S,P)$ subordinate to $V$.
	
	\smallskip 
	
$(\impliedby)$ Let $\|h(S,P)\| \le \|h\|_{\infty,V}$ for every $\Gamma$-contraction $(S,P)$ subordinate to $V$. Let $\underline{T} = (T_1,T_2,T_3) \in Q\Pe$ be subordinate to $W$.  Since $(T_1, T_2, T_3) \in Q\Pe$, $\|T_2\| \leq 2$ and $\|\Phi_\al(T_2, T_3)\| \leq 1$ for all $\alpha \in \D$. It follows from Theorem 1.5 in \cite{AglerYoung2003} that $(T_2, T_3)$ is a $\Gamma$-contraction. Moreover, $\sigma_T(\underline{T}) \subseteq W = \{0\} \times V$.
	Let $G(z^{(1)},z^{(2)},z^{(3)}) = z^{(1)}$. Clearly, $G$ is holomorphic on $\C^3$ and $G|_W = 0$. Since $\underline{T}$ is subordinate to $W$, it follows that $T_1 = G(T_1,T_2,T_3) = 0$ and so, $\underline{T} = (0,T_2,T_3)$. Next, we show that $(T_2, T_3)$ is subordinate to $V$. By spectral mapping principle, $\sigma_T(T_2, T_3) \subseteq V$. Suppose $g_0$ is holomorphic in a neighbourhood $U$ of $V$ with $g_0|_V=0$. Then the function $G_0(z^{(1)},z^{(2)},z^{(3)})= g_0(z^{(2)},z^{(3)})$ is holomorphic on $\mathbb{C} \times U$, which is a neighbourhood of $W$ and $G_0|_W=0$. Since $\underline{T}$ is subordinate to $W$, we have that $g_0(T_2,T_3) = G_0(0,T_2,T_3) = G_0(T_1,T_2,T_3) = 0$.
	Hence, $(T_2,T_3)$ is subordinate to $V$. By hypothesis,
	$
	\|f(\underline{T})\| = \|f(0,T_2,T_3)\| = \|h(T_2,T_3)\| \le \|h\|_{\infty,V} = \|f\|_{\infty,W}.
	$
By Theorem \ref{thm_ext_P}, there exists $\widehat{F} \in H^\infty(\Pe)$ such that $\widehat{F}|_W = f$ and $\|\widehat{F}\|_{\infty,\Pe} = \|f\|_{\infty,W}$. Define $\widehat{h} : \G_2 \to \C$ by $\widehat{h}(z^{(2)},z^{(3)})= \widehat{F}(0,z^{(2)},z^{(3)})$, which is holomorphic on $\G_2$. For $(z^{(2)},z^{(3)}) \in V$, we have that $\widehat{h}(z^{(2)},z^{(3)}) = \widehat{F}(0,z^{(2)},z^{(3)}) = f(0,z^{(2)},z^{(3)})= h(z^{(2)},z^{(3)})$ and thus, $\widehat{h}|_V = h$. Furthermore,
	$\|\widehat{h}\|_{\infty,\G_2} \le \|\widehat{F}\|_{\infty,\Pe} = \|f\|_{\infty,W} = \|h\|_{\infty,V}$ and so, $\|\widehat{h}\|_{\infty,\G_2} = \|h\|_{\infty,V}$. 
	\end{proof}
	
	\noindent \textbf{Funding.} The first named author is supported in part by the “Core Research Grant” with Award No. CRG/2023/005223 from Anusandhan National Research Foundation (ANRF) of Govt. of India. The second named author is supported via the IIT Bombay RDF Grant of the first named author with Project Code RI/0115-10001427.


\begin{thebibliography}{9}
		
		\vspace{0.3cm}
		
		\bibitem{Abouhajar}
		A. A. Abouhajar, M. C. White and N. J. Young, \textit{A Schwarz lemma for a domain related to $\mu$-synthesis}, J. Geom. Anal., 17 (2007), 717 -- 750.
		
		\bibitem{AbrahamseI}
		M. B. Abrahamse, \textit{The Pick interpolation theorem for finitely connected domains}, Michigan Math. J., 26 (1979), 195 -- 203. 
		
		
		
		\bibitem{Agler1990}
		J . Agler, \textit{On the representation of certain holomorphic functions defined on a polydisc}, In: Topics in operator theory: Ernst D. Hellinger memorial volume, Oper. Theory Adv. Appl., 48 (1990), 47 -- 66.
		
		\bibitem{Agler_2023}
		J. Agler, L. Kosi\'{n}ski and J. E. McCarthy, \textit{Complete norm-preserving extensions of holomorphic functions}, Israel J. Math., 255 (2023), 251 -- 263.
		
		
		\bibitem{AglerIV}
		J. Agler, Z. A. Lykova and N. J. Young, \textit{The complex geometry of a domain related to $\mu$-synthesis}, J. Math. Anal. Appl., 422 (2015), 508 -- 543. 
		
		\bibitem{Agler2019}
		J. Agler, Z. A. Lykova and N. J. Young, \textit{Geodesics, retracts, and the norm-preserving extension property in the symmetrized bidisc}, Mem. Amer. Math. Soc., 258 (2019), no. 1242, vii+108 pp.
		
		
		\bibitem{Agler_McCarthyI}
		J. Agler and J. E. McCarthy, \textit{Nevanlinna-Pick interpolation on the bidisk}, J. Reine Angew. Math., 506 (1999), 191 -- 204.
		
		
		\bibitem{Agler_McCarthy}
		J. Agler and J. E. McCarthy, \textit{Pick interpolation and Hilbert function spaces}, Grad. Stud. Math., 44, Amer. Math. Soc., Providence, RI, 2002; MR1882259. 
		
		\bibitem{Agler_McCarthy_2003}
	J. Agler and J. E. McCarthy, \textit{Norm preserving extensions of holomorphic functions from subvarieties of the bidisk},	Ann. of Math., 157 (2003), 289 -- 312.
	
		
		\bibitem{AglerYoung}
		J. Agler and N. J. Young, \textit{A commutant lifting theorem for a domain in $\C^2$ and spectral interpolation}, J. Funct. Anal., 161 (1999), 452 -- 477. 
		
		
		
		\bibitem{AglerI}
		J. Agler and N. J. Young, \textit{The two-point spectral Nevanlinna-Pick problem}, Integral Equations Operator Theory, 37 (2000), 375 -- 385.
		
		\bibitem{AglerYoung2003}
		J. Agler and N. J. Young, \textit{A model theory for $\Gamma$-contractions}, J. Operator Theory, 49 (2003), 45 -- 60. 
		
		
		\bibitem{Agler2004_II}
		J. Agler and N. J. Young, \textit{The hyperbolic geometry of the symmetrized bidisc}, J. Geom. Anal., 14 (2004), 375 -- 403.
		
		\bibitem{AglerYoung2017}
		J. Agler and N. J. Young, \textit{Realization of functions on the symmetrized bidisc}, J. Math. Anal. Appl., 453 (2017), 227 -- 240.  
		

		\bibitem{Ando}
		T. And\^{o}, \textit{On a pair of commutative contractions}, Acta Sci. Math. (Szeged), 24 (1963), 88 -- 90.
		
		\bibitem{V_Arkh}
		A. V. Arkhangel'skii and L. S. Pontryagin, \textit{General topology-$I$}, Springer, Berlin, 1990.
		
		
		\bibitem{Ball_Guerra} 
		J. A. Ball and M. D. Guerra Huam\'{a}n,	\textit{Test functions, Schur-Agler classes and transfer-function realizations: the matrix-valued setting},	Complex Anal. Oper. Theory, 7 (2013), 529 -- 575. 
		
		
		\bibitem{Ball}
		J. A. Ball and T. T. Trent, \textit{Unitary colligations, reproducing kernel Hilbert spaces, and Nevanlinna-Pick interpolation in several variables}, J. Funct. Anal., 157 (1998), 1 -- 61.
		
	
		
		\bibitem{Tirtha_Pal_Roy}
		T. Bhattacharyya, S. Pal and S. Shyam Roy, \textit{Dilations of $\Gamma$-contractions by solving operator equations}, Adv. Math., 230 (2012), 577 -- 606.
		
		
		\bibitem{Tirtha_Sau}
		T. Bhattacharyya and H. Sau, \textit{Holomorphic functions on the symmetrized bidisk- realization, interpolation and extension}, J. Funct. Anal., 274 (2018), 504 -- 524.
		
		
		\bibitem{Costara2005}
		C. Costara, \textit{The $2 \times 2$ spectral Nevanlinna-Pick problem}, J. London Math. Soc., 71 (2005), 684 -- 702. 
		

		
		
		\bibitem{Doyle}
		J. C. Doyle and G. Stein, \textit{Multivariable feedback design: concepts for a classical/modern synthesis}, IEEE Transactions on Automatic Control, 26 (1981), 4 -- 16. 
		
		
		\bibitem{Drit2007_II}
		M. A. Dritschel, S. Marcantognini and S. McCullough, \textit{Interpolation in semigroupoid algebras}, J. Reine Angew. Math., 606 (2007), 1 -- 40. 
		
		\bibitem{Drit2007_I}
		M. A. Dritschel and S. McCullough, \textit{Test functions, kernels, realizations and interpolation}, in: Operator theory, structured matrices, and dilations, Theta Ser. Adv. Math., 7 (2007), 153 -- 179. 
		
		
		
		\bibitem{Foias_Frazho}
		C. Foias and A. E. Frazho, \textit{The commutant lifting approach to interpolation problem}, Birkh\"{a}user, Berlin, 1990. 
		
		
		\bibitem{Francis}
		B. A. Francis, \textit{A course in $H_\infty$ control theory}, Lecture Notes in Control and Information Sciences, Vol. 88, London: Springer-Verlag, 1987. 
		
		
			\bibitem{Jain}
		S. Jain, S. Kumar, M. K. Mal and P. Pramanick, \textit{Function theory on tetrablock: realization, interpolation, extension and Toeplitz Corona theorem}, arXiv: 2505.23492.
		
		
		\bibitem{Jury}
		M. T. Jury, G. Knese and S. McCullough, \textit{Agler interpolation families of kernels}, Oper. Matrices, 3 (2009), 571 -- 587. 
		
		
			\bibitem{Kosinski}
		L. Kosi\'{n}ski and W. Zwonek, \textit{Nevanlinna-Pick problem and uniqueness of left inverses in convex domains, symmetrized bidisc and tetrablock}, J. Geom. Anal., 26 (2016), 1863 -- 1890. 
		
		\bibitem{Kosinski_2021}
		L. Kosi\'{n}ski and W. Zwonek, \textit{Extension property and universal sets},
		Canad. J. Math., 73 (2021), 717 -- 736.
		
				\bibitem{Mittal}
		M. Mittal, \textit{Function theory on the quantum annulus and other domains},
		Thesis (Ph.D.)-University of Houston, ProQuest LLC, Ann Arbor, MI, 2010, 141 pp.
		
		\bibitem{Nevanlinna}
		R. Nevanlinna, \textit{\"{U}ber beschr\"{a}nkte Funktionen, die in gegebenen Punkten vorgeschriebene Werte annehmen}, Ann. Acad. Sci. Fenn. Ser. A, 13 (1919), 1 -- 71. 
		
		\bibitem{Nikolov}
		N. Nikolov, P. Pflug and  P. J. Thomas, \textit{Spectral Nevanlinna-Pick and Carathéodory-Fej\'{e}r problems for $n \leq 3$}, Indiana Univ. Math. J., 60 (2011), 883 -- 893.
		
		\bibitem{Pal_penta}
		S. Pal and N. Tomar, \textit{Operators associated with the pentablock and their relations with biball and symmetrized bidisc}, Ann. Fenn. Math., 51 (2026), 287 -- 324.
		
		
		\bibitem{Pick} 
		G. Pick, \textit{\"{U}ber die Beschr\"{a}nkungen analytischer Funktionen, welche durch vorgegebene Funktionswerte bewirkt werden}, Math. Ann., 77 (1916), 7 -- 23. 
		
			\bibitem{Scheidemann}
		V. Scheidemann, \textit{Introduction to complex analysis in several variables}, Birkh\"{a}user Verlag, Basel, 2005. 
		
		\bibitem{Vasilescu}
		F. -H. Vasilescu,\textit{ Analytic functional calculus and spectral decompositions}, 1 (1982), Math. Appl. (East European Ser.), D. Reidel Publishing Co., Dordrecht. 
		
		
	\end{thebibliography}
\end{document}